\documentclass[leqno,final]{article}
\usepackage{setspace} 
\usepackage{geometry}
\usepackage{fancyhdr}
\usepackage{cite}
\usepackage{graphicx}
\usepackage{epstopdf}
\usepackage{subfig}
\usepackage{tikz}
\usetikzlibrary{positioning, shapes.geometric}
\usetikzlibrary{decorations.markings}
\usetikzlibrary{arrows.meta}

\usepackage{color}
\usepackage{xcolor}
\usepackage{hyperref}

\hypersetup{hidelinks}
\usepackage{indentfirst}
\usepackage{url}
\usepackage{amsfonts,amsmath,amssymb,amsthm,latexsym,lineno}
\usepackage{enumerate}
\usepackage{bm}
\usepackage{algorithm,algorithmic}
\usepackage[normalem]{ulem}
\usepackage{cleveref}

\ifpdf
 \DeclareGraphicsExtensions{.eps,.pdf,.png,.jpg}
\else
 \DeclareGraphicsExtensions{.eps}
\fi

\newtheorem{theorem}{Theorem}[section]
\newtheorem{lemma}[theorem]{Lemma}
\newtheorem{remark}[theorem]{Remark}
\newtheorem{example}[theorem]{Example}

\numberwithin{theorem}{section}
\numberwithin{equation}{section}
\numberwithin{figure}{section}
\numberwithin{algorithm}{section}



\newcommand{\be}{\begin{equation}}
\newcommand{\ee}{\end{equation}}
\newcommand{\ba}{\begin{array}}
\newcommand{\ea}{\end{array}}
\newcommand{\ben}{\begin{eqnarray}}
\newcommand{\een}{\end{eqnarray}}
\newcommand{\bn}{\begin{eqnarray*}}
\newcommand{\en}{\end{eqnarray*}}

\newcommand{\bbn}{{\bf n}}

\newcommand{\cTh}{{\cal T}_h}

\newcommand{\ls}{\lesssim}

\newcommand{\al}{\alpha}

\newcommand{\ga}{\gamma}
\newcommand{\Ga}{\Gamma}

\newcommand{\na}{\nabla}
\newcommand{\om}{\omega}
\newcommand{\Om}{\Omega}
\newcommand{\pa}{\partial}

\newcommand{\si}{\sigma}

\newcommand{\vp}{\varphi}

\newcommand{\norm}[1]{\left\Vert#1\right\Vert}

\newcommand{\abs}[1]{\left\vert#1\right\vert}

\newcommand{\set}[1]{\left\{#1\right\}}

\newcommand{\eq}[1]{\begin{align}#1\end{align}}
\newcommand{\eqn}[1]{\begin{align*}#1\end{align*}}

\DeclareSymbolFont{ugmL}{OMX}{mdugm}{m}{n}
\SetSymbolFont{ugmL}{bold}{OMX}{mdugm}{b}{n}
\DeclareMathAccent{\wideparen}{\mathord}{ugmL}{"F3}

\title{The effect of numerical integration in the FEM for elliptic problems with mixed boundary conditions}

\author{
Juan Han\thanks{School of Mathematics, Nanjing University, Nanjing 210093, China.
Email address: dg20210004@smail.nju.edu.cn}
\,\,\,\mbox{and}\,\,\,
Haijun Wu\thanks{School of Mathematics, Nanjing University, Nanjing 210093, China. 
Email address: hjw@nju.edu.cn}
}

\date{}
\begin{document}
	\maketitle

	% REQUIRED
	\begin{abstract}
		This paper investigates the impact of quadrature accuracy for volume and face integrals in the finite element method using $p$-th order polynomial shape functions for elliptic problems with mixed Dirichlet and Robin boundary conditions. The optimal $p$-th order $H^1$-convergence is maintained when numerical integration with algebraic precision at least  $2p-2$ for volume terms and at least $2p-1$ for face terms is adopted. For $L^2$-error, we achieve optimal $O(h^{p+1})$ convergence when using quadrature rules of precision no less than $\max\{p,2p-2\}$ for volume terms and no less than $2p-1$ for face terms. Of particular significance, we present two examples to show that the above result on $L^2$-error is sharp for the linear FEM ($p=1$). When reduced to the case of Dirichlet boundary condition, our results yield improved dependence on the given data compared to the classical results established by Ciarlet, \textit{et al}.  Numerical experiments are provided to illustrate the theoretical findings and confirm the necessity of specified quadrature accuracy and data regularity.
	\end{abstract}
	

\section{Introduction}\label{Introduction}

The finite element method (FEM) is implemented by selecting a set of basis functions for the variational problem, forming a system of equations, and then solving this system. The formation of these systems involves calculating numerous integrals inherent in the variational formulation. In many practical scenarios, exact computation of these integrals is not feasible, so numerical integration is needed to approximate them. What conditions must numerical integration satisfy to ensure it does not compromise the convergence order of the errors? Are these conditions sharp?

For the first question, mainstream engineering simulation platforms (e.g., ANSYS \cite{ansys}, COMSOL \cite{comsol} and FEniCS \cite{Logg12}, etc.) typically require numerical integration schemes to exactly evaluate the integrals without coefficients $\int_K v_h w_h$. Specifically, when employing piecewise polynomials of order $p$, these implementations conventionally adopt quadrature rules with algebraic precision of $2p$, as stated in \cite{comsol}: ``{\itshape ... For smooth constraints, a sufficient integration order is typically twice the order of the shape function ... }''.
While this empirical choice proves effective for handling second-order PDE systems commonly encountered in engineering applications, we note that this condition represents a conservative estimate rather than the minimal theoretical requirement, i.e., it is not sharp.

Although the control of the effect introduced by numerical integration is important for almost all applications of FEMs to problems in engineering and sciences, research in this area is meager compared to the theory of FEM itself. For the second-order elliptic equations $-\na\cdot(\al\na u)=f$ with homogeneous Dirichlet boundary conditions, Herbold  initially explored one-dimensional \cite{Herbold69} and multi-dimensional rectangular \cite{Herbold71} domains. Later, Ciarlet's classical arguments in \cite{Ciarlet1991,Ci78} established a general framework over polygonal domains, where the finite element formulation involves only volume integrals. A concrete manifestation of this theory shows that for $p$-th order elements, quadrature formulas with algebraic precision of at least $2p-2$ suffice to preserve the $H^1$-convergence rate. The analysis for $L^2$-norm error estimates was established by \cite{CR72}, combining the effect of curved boundaries and numerical integration, proving that quadrature precision of $\max\{p, 2p-2\}$ similarly maintains the FEM's optimal convergence rate. This condition of algebraic precision for $L^2$-error is the same as that for $H^1$-error when $p\ge 2$, but is one order higher for linear FEM (p=1). Babu\v{s}ka \cite{bbl11} then examined the impact of numerical integration on finite element approximations of linear functionals, by selecting an appropriate functional, one can evaluate the result concerning the $H^1$-error. However, it is worth noting that the estimate in \cite{bbl11} relies on $\|f\|_{p+1,\Omega}$ rather than $\|f\|_{p,q,\Omega}$ as utilized in \cite{Ci78}, necessitating the algebraic precision of $2p-1$. While it is possible to derive results concerning the $L^2$-error by selecting suitable functionals, they may not be optimal.

There have also been relevant studies on other model problems. Such results mainly have been derived by Raviart \cite{Raviart73} for parabolic equations, by Baker and Dougalis \cite{BakerDougalis76} for second order hyperbolic equations, by Banerjee \cite{Banerjee92,BanerjeeOsborn90} for finite element eigenvalue approximation, by Aylwin and Jerez-Hanckes \cite{AylwinJerez21} for Maxwell variational problems, and by Abdule and Vilmart \cite{AbdulleVilmart12} for nonmonotone nonlinear elliptic problems. Moreover, \cite{KimSuri93} and \cite{sobotikova08} analyzed the $p$-version of the finite element method in the presence of numerical integration.

Turning back to the basic model problem of second order elliptic equations, a few issues still need to be addressed. First, somewhat surprisingly, the authors found no literature dealing with Robin boundary condition (which is certainly one of the most important boundary conditions), for which the FEM involves  integrals on faces of element. Second, it remains to be examined whether the condition of algebraic precision for $L^2$-error of the linear FEM is necessary. Third, it is worth investigating whether the regularity requirements on the exact solution in the classical results in \cite{Ciarlet1991,Ci78,CR72} be weakened. In this paper, we consider the elliptic equation $-\na\cdot(\al\na u)+\beta u=f$ with mixed Dirichlet and Robin boundary condition. 
 Following the approach in \cite{CR72}, we conduct a rigorous analysis of how numerical integration affects both $H^1$- and $L^2$-error estimates of FEM on triangulations of polygonal domains. Specifically, we establish that achieving optimal convergence requires volume integration precision of $2p-2$ for the $H^1$-error and $\max\{p, 2p-2\}$ for the $L^2$-error. As for face integration, our analysis reveals that taking algebraic precision of $2p-1$ in boundary integration is sufficient to preserve the optimal convergence rates for both $H^1$- and $L^2$-norms, forming a uniform condition that holds consistently across all polynomial degrees $p$. Moreover, our examination provides finer-grained insights into its dependence on data regularity, compared to classical results \cite{Ci78,CR72,bbl11} for the case of pure Dirichlet boundary condition. 

We further provide concrete examples to demonstrate the sharpness of these conditions. Notably, for $p=1$, the $L^2$-error requires one order higher quadrature precision in volume integration compared to $H^1$-error. Therefore, we especially present one-dimensional examples with mixed and Dirichlet boundary conditions, where the analysis of $L^2$-error lower bounds under a quadrature of zero-order precision highlights the necessity of employing quadratures of first-order precision. Additional cases are systematically examined through numerical experiments. In summary, to ensure optimal order convergences, the algebraic precision of the quadrature rule for volume integration required for $L^2$-error is one order higher than that for the $H^1$-error for the linear FEM, and the same for higher order FEM ($p>1$). The algebraic precisions requirements of the quadrature rule for face integration are the same for both errors for each $p\ge 1$.

The rest of our paper is organized as follows.  Section~\ref{Preliminaries} establishes the notation and theoretical foundations. Section~\ref{errorEstimation} analyzes the impact of numerical integration on both $H^1$- and $L^2$-error estimates for mixed boundary value problems. Section~\ref{lowerBounds} provides two concrete examples ($p=1$) to demonstrate the sharpness of the  precision requirements for volume integration. Finally, Section~\ref{numericalExperiment} presents numerical experiments validating our theoretical results.

Throughout the paper, $C$ is used to denote a generic positive constant
which is independent of the mesh size $h$, the data and the exact solution of the elliptic problem. We also use the shorthand notation $A\lesssim B$ and $B\gtrsim A$ for the inequality $A\leq C B$ and $B\geq CA$. $A\eqsim B$ is for the statement $A\lesssim B$ and $B\lesssim A$. 

%  \cb{\begin{theorem}\label{ciarlet}
%     Suppose $E_{\hat{K}}(\hat{\varphi})=0~\forall\hat{\varphi}\in P_{2p-2}(\hat{K})$ and let $\al,\beta\in W^{p,\infty}(\Om)$ and $f\in W^{p,q}(\Om)$, here $q\ge2$ with $p>\frac{n}{q}$. If $u\in H^{p+1}(\Om)$, then we have
%     \begin{align}
%       \norm{u-u^{\star}_h}_{1,\Om}\lesssim h^p\big(\abs{u}_{p+1,\Om}+(\norm{\al}_{p,\infty,\Om}+\norm{\beta}_{p,\infty,\Om})\norm{u}_{p,\Om}+\norm{f}_{p,q,\Om}\big).
%     \end{align}
%   \end{theorem}}

\section{Preliminaries}\label{Preliminaries} 

Let $\Omega$ be a bounded convex polygonal domain with boundary $\Gamma$ in $\mathbb{R}^n$, where $n=1,2,3$, and $\Gamma=\Gamma_1\cup\Gamma_2$, $\Gamma_1\cap\Gamma_2=\emptyset$. We adopt standard notation for function spaces, norms, and inner products (see, e.g., \cite{bs08,Ci78}). Specifically,
$(\cdot,\cdot)$ and $\langle \cdot,\cdot \rangle$ denote the $L^2$-inner products on the domain $\Omega$ and boundary portion $\Gamma_2$, respectively. $|\cdot|_{j,\om}$ represents the usual semi-norm of $H^{j}(\om)$, while $\|\cdot\|_{j,\om}$ denotes the corresponding norm. For $q \neq 2$, the notations $|\cdot|_{j,q,\omega}$ and $\|\cdot\|_{j,q,\omega}$ denote the semi-norm and norm in $W^{j,q}(\omega)$. Besides, let $P_p(\omega)$ denote the space of polynomials of degree at most $p$ defined on a domain $\omega$.

We consider the effect of numerical integration in the  FEM for the following elliptic problems with mixed Dirichlet and Robin boundary conditions: 
  \eq{
    \label{EP}
    \begin{cases}
      -\nabla\cdot(\al\nabla u) + \beta u = f \quad &{\rm in}\quad \Om,\\
      u=0\quad &{\rm on} \quad\Ga_1,\\
      \al\nabla u\cdot\bbn + \gamma u= g \quad &{\rm on} \quad\Ga_2,
    \end{cases}
  }
  where $\al_0\le\al(x)\le \al_1$, $\beta_0\le \beta(x)\le \beta_1$ and $\gamma(x)\ge 0$ hold for some constants $\al_0,\al_1, \beta_0,\beta_1>0$.
%Where $\Omega$  is a a bounded and convex polygonal domain in $\mathbb{R}^n$, $n=1,2,3$. $\p \Omega=\Gamma_1\cup\Gamma_2$ and $\Gamma_1\cap\Gamma_2=\emptyset$.

\subsection{The FEM}

The variational form of \eqref{EP} is: find $u\in V:=\set{v\in H^1(\Omega):\;v|_{\Ga_1}=0}$, s.t.
  \be\label{variationalForm2}
    a(u,v)=F(v)\quad  \forall v\in V,
  \ee
  where	
  \[a(u,v):=(\al\nabla u,\nabla v)+(\beta u,v)+\langle \gamma u,v \rangle,\quad F(v):=(f,v)+\langle g,v \rangle.\]
Clearly, the variational problem \eqref{variationalForm2}  has a unique solution for any $f\in V'$ and $g\in L^2(\Ga_2)$, where $V'$ denotes the daul space of $V$. 

Let $\mathcal{T}_h$ denote a regular and quasi-uniform triangulation of $\Omega$, and $\mathcal{F}_h$ denote the set of element faces on $\Ga_2$. In this paper, we distinguish between integrals involving $\mathcal{T}_h$ and $\mathcal{F}_h$ using volume and face integrals, respectively.
Given an integer $p\ge 1$, introduce the finite element spaces,
\eqn{V_h:=\set{v_h\in C(\bar\Om):~v_h|_{K}\in P_p(K)\;  \forall K\in\cTh,v_h|_{\Ga_1}=0}.} The FEM for the variational problem \eqref{variationalForm2} reads: find $u_h \in V_h$, s.t.
  \be\label{FEMformNeumann}
    a(u_h,v_h)=F(v_h)\quad \forall v_h\in V_h.
  \ee
In fact, if the solution $u$ is smooth enough to belong to the space $H^{p+1}(\Om)$, we have the following optimal order $H^1$- and $L^2$-error estimates \cite{Ci78}:
  \be\label{neumann-uh-error}
    |u-u_{h}|_{j,\Om}\le Ch^{p+1-j}|u|_{p+1,\Om},\quad j=0,1.
  \ee
For further analysis, given $s\ge 0$, we introduce the following discrete $H^s$ norms:
\eqn{\norm{\cdot}_{s,\mathcal{T}_h}:=\Big(\sum_{K\in \mathcal{T}_h}\norm{\cdot}_{s,K}^2\Big)^\frac12\quad\text{and}\quad \norm{\cdot}_{s,\mathcal{F}_h}:=\Big(\sum_{e\in \mathcal{F}_h}\norm{\cdot}_{s,e}^2\Big)^\frac12.}

\subsection{Numerical integration}\label{NumericalIntegration}

Let $\hat{K}$ be a reference element affine equivalent to any element $K \in \mathcal{T}_h$ via an invertible affine mapping $F_K: \hat{K} \to K$. And the set $\{\hat w_{l},\hat b_{l}\}_{l=1}^{L}$ with weights $\hat w_{l}>0$ and nodes $\hat b_{l}\in\hat K$ defines a quadrature rule on the reference element $\hat K$,
\eq{\label{qd_hK}\int_{\hat{K}}\hat{\varphi}(\hat{x}){\,\rm d}\hat{x}\sim\sum_{l=1}^{L}\hat{w}_l \hat{\varphi}(\hat{b}_l).}
Then the corresponding quadrature rule on $K$ is given by
\eqn{\int_{K}\varphi(x){\,\rm d}x\sim\sum_{l=1}^{L} w_{l,K}\varphi(b_{l,K}),}
where $\hat\varphi=\varphi\circ F_K$, $w_{l,K}=\frac{|K|}{|\hat K|}\hat{w}_l,~\text{and}~b_{l,K}=F_K(\hat{b}_l)$. In the same way, let $\hat e$ be a reference face, which is affine equivalent to any face $e$ of an element $K$ under an invertible affine mapping $F_e: \hat e\mapsto e$, and let the set $\{\hat \omega_{m},\hat{\mathsf{b}}_{m}\}_{m=1}^{M}$ with weights $\hat \omega_{m}>0$ and nodes $\hat{\mathsf{b}}_{m}\in\hat e$, determine a quadrature rule on the reference face $\hat e$. 
\eq{\label{qd_he}\int_{\hat{e}}\hat{\varphi}(\hat{x}){\,\rm d}\hat{\si}\sim\sum_{m=1}^{M}\hat\omega_{m} \hat{\varphi}(\hat{\mathsf{b}}_{m}).}
Then the set $\{{\omega}_{m,e},b_{m,e}\}_{m=1}^{M}$ with $w_{m,e}=\frac{|e|}{|\hat e|}\omega_{\hat e}, b_{m,e}=F_e(\hat{\mathsf{b}}_{m})$ determines the corresponding  quadrature rule on face $e$:
\eqn{\int_{e}\varphi(x){\,\rm d}\si\sim\sum_{m=1}^{M} w_{m,e}\varphi(b_{m,e}).}
Consequently, the FEM with numerical integration is given by: find $u_h^{\star}\in V_h$, s.t.
  \be\label{integrationFEM}
  a^{\star}(u_h^{\star},v_h)=F^{\star}(v_h)\quad \forall v_h\in V_h,
  \ee
  where 
  \begin{align*}
  a^{\star}(u_h,v_h)&=\sum_{K\in\cTh}\sum_{l=1}^{L} w_{l,K}\Big[\al\na u_h\cdot\na v_h+\beta u_hv_h\Big](b_{l,K})+\sum_{e\in\mathcal{F}_h}\sum_{m=1}^{M} {w}_{m,e} [\gamma u_h v_h](b_{m,e}),\\
  F^{\star}(v_h)&=\sum_{K\in\cTh}\sum_{l=1}^{L} {w}_{l,K} [fv_h](b_{l,K}) +\sum_{e\in\mathcal{F}_h}\sum_{m=1}^{M} {w}_{m,e} [gv_h](b_{m,e}).
\end{align*}

We introduce the quadrature error functionals,
  \[E_K(\varphi)=\int_{K}\varphi(x){\,\rm d}x-\sum_{l=1}^{L} w_{l,K}\varphi(b_{l,K})\quad\text{and}\quad E_{\hat K}(\hat{\varphi})=\int_{\hat{K}}\hat{\varphi}(\hat{x}){\,\rm d}\hat{x}-\sum_{l=1}^{L}\hat{w}_l \hat{\varphi}(\hat{b}_l).\]
Similarly, we continue to define quadrature error functionals on each face $e\in\mathcal{F}_h$ and the reference face $\hat e$ as follows:
\[E_e(\varphi)=\int_{e}\varphi(x){\,\rm d}\sigma-\sum_{m=1}^{M} w_{m,e}\varphi(b_{m,e}),~\text{and}~E_{\hat{e}}(\hat{\varphi})=\int_{\hat{e}}\hat{\varphi}(\hat{x}){\,\rm d}\hat\sigma-\sum_{m=1}^{M}\hat{\omega}_m \hat{\varphi}(\hat{\mathsf{b}}_{m}).\]
It is evident that  
\eq{\label{EKhEK} E_K(\varphi)=\frac{|K|}{|\hat K|} E_{\hat K}(\hat{\varphi})\quad\text{and}\quad E_e(\varphi)=\frac{|e|}{|\hat{e}|}E_{\hat{e}}(\hat{\varphi}).}
Recall that the algebraic precision of the quadrature rule is defined by the largest integer $p$ such that $\forall v\in P_p(\omega)$, $E_\om(v)=0$, where $\omega=e$ or $K$. We remark that no boundary integrals are involved for the case \(n=1\). Accordingly, the algebraic precision of boundary quadratures is defined for two- and three-dimensional problems only. For simplicity of presentation, we will not distinguish between different dimensions in the subsequent discussion.

% Furthermore, 
% where $w_{m,e}=\frac{|e|}{|\hat{e}|}\hat{w}_m$.		In this case, we have the scale relationship $$.

To estimate the errors of the quadrature rules, we need the following lemma which is a special case of the Bramble-Hilbert lemma \cite[Theorem 4.1.3]{Ci78}.
\begin{lemma}\label{lemBL} Let $\om\subset\mathbb{R}^n$ be a domain with Lipschitz boundary, $m\ge 1$, $1\le q\le\infty$, and $E$ be a bounded linear functional on $W^{m,q}(\om)$, which satisfies $E v=0\;\forall v\in P_{m-1}$. Then
\eqn{\abs{E \vp}\ls \abs{\vp}_{m,q,\om}\quad\forall \vp\in W^{m,q}(\om).} 
\end{lemma}

% A quadrature rule for an integral on $\om$ is represented by
% \eq{\label{QDo}\int_\om\vp(x) dx\sim \sum_{l=1}^{L}w_l \varphi(b_l),}
% where $w_l$ and $b_l$ are called weights and nodes, respectively. Denote the error functional as
% \eqn{E_\om(\vp)=\int_{\om}\vp(x)dx-\sum_{l=1}^{L} w_l\varphi(b_l).}
% Clearly, $E$ is bounded on $C^0(\bar\om)$, i.e., $|E(\vp)|\ls\norm{\vp}_{C^0(\bar\om)}\;\forall \vp\in C^0(\bar\om)$.
% Recall that the algebraic precision of the quadrature rule is defined by the largest integer $p$ such that $E_\om(v)=0\;\forall v\in P_p$. 
% \begin{lemma}\label{lemQD} Let $\om\subset\mathbb{R}^n$ be a domain with Lipschitz boundary, $m\ge 1$, $1\le q\le\infty$. If $m-\frac{n}{q}>0$ and the quadrature formula \eqref{QDo} has algebraic precision $m-1$, then
% \eqn{|E_\omega(\vp)|\ls \abs{\vp}_{m,q,\om}\quad\forall \vp\in W^{m.q}(\om).} 
% \end{lemma}
% \begin{proof}
% From the Sobolev embedding theorem \cite{adm75}, we have $|E(\vp)|\ls \norm{\vp}_{C^0(\bar\om)}\ls \norm{\vp}_{m,q,\om}$. Then the proof follows by using the Bramble-Hilbert lemma.
% \end{proof}

Additionally, we will frequently employ the following scaling results in our subsequent analyses.
\begin{lemma}\label{lemScale} Suppose $\om, \hat\om\subset\mathbb{R}^n$ are affine equivalent, that is, there exists an invertible affine mappings $ F:\hat\om\rightarrow \om$. Let $m\ge 1$ and $1\le q\le\infty$. Then, for any function $v\in W^{m,q}(\om)$, we have  $\hat v:=v\circ F\in W^{m,q}(\hat\om)$ and 
\eqn{|\hat{v}|_{m,q,\hat\om}&\ls h_\om^m|\om|^{-1/q}|v|_{m,q,\om},\quad |v|_{m,q,\om}\ls\rho_\om^{-m}|\om|^{1/q}|\hat{v}|_{m,q,\hat\om},
}
where $h_\om$ is the diameter of $\om$, $\rho_\om$ is the diameter of the largest ball contained in $\om$, and the invisible constants after $\ls$ may depends on $\hat\om$ but independent of $\om$.
\end{lemma}

%Before delving into the approximation properties of $u_{h}^{\star}$, we must ensure that problem \eqref{integrationFEM} has a unique solution, which necessitates that $a^{\star}(\cdot,\cdot)$ possesses uniform ellipticity as proposed in \cite{Ci78}. In the following lemma, we provide the sufficient conditions for $a^{\star}$ to possess uniform $V_h$-ellipticity.

The well-posedness of the discrete problem \eqref{integrationFEM} is a consequence of the following lemma which gives the uniform $V_h$-ellipticity of the bilinear form $a^{\star}(\cdot,\cdot)$.  
\begin{lemma}\label{Vh-ellipticity}
  Let there be given a quadrature scheme $E_{\hat K}(\cdot)$ that is exact for the space $P_{2p-2}(\hat{K})$. Then the approximate bilinear form satisfies uniform $V_h$-ellipticity, i.e.
  \be
    a^{\star}(v_h,v_h)\gtrsim \norm{v_h}^2_{1,\Om}\quad \forall v_h\in V_h.
  \ee
\end{lemma}
\begin{proof}
The proof follows by using the result \cite[Theorem~4.1.2]{Ci78} for the problem with Dirichlet boundary condition and the fact that $\sum_{e\in\mathcal{F}_h}\sum_{m=1}^{M} {w}_{m,e} [\gamma v_h v_h](b_{m,e})\ge 0$.
\end{proof}

This lemma is also crucial in the subsequent error analysis, and it is correct for $\Gamma_1=\emptyset$ or $\Gamma_2=\emptyset$.

\section{Error estimates of FEM with numerical integration}\label{errorEstimation}

In this section, we will establish Strang-type lemmas for both the $H^1$- and $L^2$-norm error estimates by exploiting the uniform $V_h$-ellipticity from Lemma~\ref{Vh-ellipticity} and the dual problem. Through systematic analysis of each component of the consistency error, we will ultimately derive the error estimates for the discrete problem \eqref{integrationFEM}.

\subsection{$\mathbf{H^1}$-error estimation}\label{errorEstimationH1}
Similar to the Céa lemma, we now present the first Strang lemma, which demonstrates that the error of the discrete solution to \eqref{integrationFEM} can be  managed by considering the interpolation error and the numerical integration approximation errors of $a(\cdot,\cdot)$ and $F(\cdot)$. Since the proof of this lemma proceeds analogously to that of \cite[Theorem 4.1.1]{Ci78}, here we omit the details of the proof.

\begin{lemma}\label{Strang}
  If $a^{\star}(\cdot,\cdot)$ is uniformly $V_h$-elliptic, then we have
  \eq{\label{Strang1}\norm{u-u_{h}^{\star}}_{1,\Om}&\lesssim\inf_{v_h\in V_h}\bigg(\norm{u-v_h}_{1,\Om}+\sup_{0\neq w_h\in V_h}\frac{|a(v_h,w_h)-a^{\star}(v_h,w_h)|}{\norm{w_h}_{1,\Om}}\bigg)\notag\\
  &\quad +\sup_{0\neq w_h\in V_h}\frac{|F^{\star}(w_h)-F(w_h)|}{\norm{w_h}_{1,\Om}}.}
\end{lemma}
% \begin{proof}
%   \cb{Similar to \cite[Theorem 4.1.1]{Ci78}}, let $v_h$ be an arbitrary element in the space $V_h$. With the assumption of uniform $V_h$-ellipticity,
%   \eqn{\norm{u_{h}^{\star}-v_h}_{1,\Om}^2&\ls a^{\star}(u_{h}^{\star}-v_h,u_{h}^{\star}-v_h)\nonumber\\
%     &=a(u-v_h,u_{h}^{\star}-v_h)+(a(v_h,u_{h}^{\star}-v_h)-a^{\star}(v_h,u_{h}^{\star}-v_h))\nonumber\\
%     &\quad+(F^{\star}(u_{h}^{\star}-v_h)-F(u_{h}^{\star}-v_h)), 
% }
%   so that, with the continuity of the bilinear form $a(\cdot,\cdot)$,
%   \begin{align*}
%     \norm{u_{h}^{\star}-v_h}_{1,\Om}
% %&\ls \norm{u-v_h}_{1,\Om}+\frac{|a(v_h,u_h^{\star}-v_h)-a^{\star}(v_h,u_h^{\star}-v_h)|}{\norm{u_h^{\star}-v_h}_{1,\Om}}+\frac{|F^{\star}(u_h^{\star}-v_h)-F(u_h^{\star}-v_h)|}{\norm{u_h^{\star}-v_h}_{1,\Om}}\\
%     \ls &\norm{u-v_h}_{1,\Om}+\sup_{0\neq w_h\in V_h}\frac{|a(v_h,w_h)-a^{\star}(v_h,w_h)|}{\norm{w_h}_{1,\Om}}\\
%     &+\sup_{0\neq w_h\in V_h}\frac{|F^{\star}(w_h)-F(w_h)|}{\norm{w_h}_{1,\Om}},
%   \end{align*}
%   which together with the triangle inequality completes the proof of the lemma.
% \end{proof}

It is clear that from Lemma \ref{Strang}, the effect of numerical integration on $H^1$-error depends on $F^{\star}(w_h)-F(w_h)$ and $a(v_h,w_h)-a^{\star}(v_h,w_h)$. In fact, if the integration is exact, these differences are zeros. To estimate these terms, we recall two estimates given in \cite{Ci78} on $E_K(\al\nabla v\cdot\nabla w)$ and $E_K(fw)$, respectively, and present an estimate on $E_K(\beta vw)$.
\begin{lemma}\label{apvpvpw-fw}
  Assume that $\forall \hat{\varphi}\in P_{2p-2}(\hat{K}),$ $E_{\hat K}(\hat{\varphi})=0.$ Then for any $K\in\cTh$, $\al,\beta\in W^{p,\infty}(K)$ and $v,w\in P_p(K)$
  \begin{align}
    |E_K(\al\nabla v\cdot\nabla w)|&\ls h_K^p\norm{\al}_{p,\infty,K}\norm{v}_{p,K}|w|_{1,K},\label{EKagvgw}\\
    |E_K(\beta vw)|&\lesssim h_K^p\norm{\beta}_{p,\infty,K}\norm{v}_{p,K}\norm{w}_{1,K}.\label{EKbvw}
  \end{align}
  Furthermore, if $q\in[1,\infty] $ is any number satisfying $p-\frac{n}{q}>0$, then 
  \eq{\label{EKfw}|E_K(fw)|\ls h_K^p|K|^{\frac12-\frac1q}\norm{f}_{p,q,K}\norm{w}_{1,K} \quad\forall\,f\in W^{p,q}(K)~\text{and}~w\in P_p(K).}
\end{lemma}
\begin{proof} For \eqref{EKagvgw} and \eqref{EKfw}, we refer to Theorems 4.1.4 and 4.1.5 in \cite{Ci78}, respectively. 

Next we prove \eqref{EKbvw} by following the arguments in the proof of \cite[Lemma~6.1]{BanerjeeOsborn90} or \cite[Theorem 4.1.5]{Ci78}. 
Let $\hat{\varPi}_0$ be the orthogonal projection from $L^2(\hat{K})$ onto $P_0(\hat{K})$. For any $v,w\in P_{p}(K)$, we have
\[ E_{\hat K}(\hat{\beta}\hat{v}\hat{w})= E_{\hat K}(\hat{\beta}\hat{v}\hat{\varPi}_0\hat{w})+ E_{\hat K}(\hat{\beta}\hat{v}(\hat{w}-\hat{\varPi}_0\hat{w})).\]
Since $ E_{\hat K}(\cdot)$ vanishes over the space $P_{p-1}(\hat{K})$ and $\norm{\hat{\varPi}_0\hat{w}}_{0,\hat{K}}\le\norm{\hat{w}}_{0,\hat{K}}$, from Lemma~\ref{lemBL} and the equivalence of norms on finite dimensional spaces, we have
\begin{align*}
| E_{\hat K}(\hat{\beta}\hat{v}\hat{\varPi}_0\hat{w})|&\lesssim |\hat{\beta}\hat{v}\hat{\varPi}_0\hat{w}|_{p,\infty,\hat{K}}\lesssim|\hat{\beta}\hat{v}|_{p,\infty,\hat{K}}\norm{\hat{\varPi}_0\hat{w}}_{0,\hat{K}}\lesssim\hspace{-0.1cm}\bigg(\sum_{j=0}^{p}|\hat{\beta}|_{p-j,\infty,\hat{K}}|\hat{v}|_{j,\hat{K}}\bigg)\norm{\hat{w}}_{0,\hat{K}}.
\end{align*}
It is clear that $E_{\hat K}(\hat{\beta}\hat{v}(\hat{w}-\hat{\varPi}_0\hat{w}))=0$ when $p=1$. For $p\ge 2$, given $\hat{w}\in P_{p}(\hat{K})$,  the linear form $ E_{\hat K}((\hat{w}-\hat{\varPi}_0\hat{w})(\cdot))$ is bounded on $W^{p-1,\infty}(\hat{K})$: $| E_{\hat K}((\hat{w}-\hat{\varPi}_0\hat{w})(\cdot))|\ls  \norm{\hat{w}-\hat{\varPi}_0\hat{w}}_{0,\hat K}\norm{\cdot}_{p-1,\infty,\hat K}$, and it vanishes over $P_{p-2}(\hat{K})$. By Lemma~\ref{lemBL}, we have
\eqn{| E_{\hat K}(\hat{\beta}\hat{v}(\hat{w}-\hat{\varPi}_0\hat{w}))|\ls |\hat{\beta}\hat{v}|_{p-1,\infty,\hat{K}}\|\hat{w}-\hat{\varPi}_0\hat{w}\|_{0,\hat{K}}\lesssim\bigg(\sum_{j=0}^{p-1}|\hat{\beta}|_{p-1-j,\infty,\hat{K}}|\hat{v}|_{j,\hat{K}}\bigg)|\hat{w}|_{1,\hat{K}}.}
Therefore, by the scaling properties in \eqref{EKhEK} and Lemma~\ref{lemScale}, we conclude that
\eqn{&|E_K(\beta vw)|\lesssim |K|\abs{E_{\hat K}(\hat{\beta}\hat{v}\hat{w})}\\
\ls &h^p\norm{\beta}_{p,\infty,K}\norm{v}_{p,K}\norm{w}_{0,K}+h^p\norm{\beta}_{p-1,\infty,K}\norm{v}_{p-1,K}|w|_{1,K},}
which implies that \eqref{EKbvw} holds. This completes the proof of the lemma.
\end{proof}
\begin{remark}
Similar results on $|E_K(\al\nabla v\cdot\nabla w)|$ are given in \cite[Lemma~3.4]{bbl11} and  \cite[Lemma~6.1]{BanerjeeOsborn90}  by assuming that $E_{\hat K}(\hat{\varphi})=0\;\forall \hat{\varphi}\in P_{2p-1}(\hat{K}).$ 
\end{remark}

Furthermore, we also need to account for the influence of numerical integration on $\Ga_2$. 
\begin{lemma}\label{lem:gw}
  Suppose that $E_{\hat{e}}(\hat{\varphi})=0\;\forall\hat{\varphi}\in P_{2p-1}(\hat{e})$ and  $r\in [1,\infty]$ satisfies $p-\frac{n-1}{r}>0$.  Then for any face $e\in\mathcal{F}_h$ and $v,w\in P_p(e).$
  \begin{align}
  |E_e(\ga vw)|&\lesssim h_e^{p}\norm{\gamma}_{p,\infty,e}\norm{v}_{p,e}\norm{w}_{0,e}\quad \forall\gamma\in W^{p,\infty}(e),\label{gw1b-1}\\
  |E_e(gw)|&\lesssim h_e^p|e|^{\frac12-\frac1{r}}\norm{g}_{p,r,e}\norm{w}_{0,e}\quad\forall  g\in W^{p,r}(e).\label{gw1b}
  \end{align}
\end{lemma}
\begin{proof} \eqref{gw1b-1} is proved in \cite[Lemma~3.2]{BanerjeeOsborn90}. 
	We only prove \eqref{gw1b}.  Thanks to the Sobolev embedding theorem (see e.g. \cite[(1.4.6)]{bs08}), $W^{p,r}(\hat{e})\hookrightarrow$ $ C^0(\hat{e})$ holds, where ``$\hookrightarrow$" denotes ``continuously embedding".  We have
\eqn{|E_{\hat{e}}(\hat{g}\hat{w})|\le\|\hat{g}\|_{0,\infty,\hat{e}}\|\hat{w}\|_{0,\infty,\hat{e}}\le||\hat{g}||_{p,r,\hat{e}}\|\hat{w}\|_{0,\infty,\hat{e}}.}
Thus for a fixed $\hat{w}\in P_p(\hat{e})$, the linear functional $E_{\hat{e}}:~\hat{g}\rightarrow E_{\hat{e}}(\hat{g}\hat{w})$ is bounded on $W^{p,r}(\hat e)$ and vanishes over the space $P_{p-1}(\hat{e})$. Now using the Lemma~\ref{lemBL} and the equivalence of norms on finite dimensional spaces, we can establish the following inequality:
	\eqn{|E_{\hat{e}}(\hat{g}\hat{w})|\lesssim|\hat{g}|_{p,r,\hat{e}}\|\hat{w}\|_{0,\hat{e}}.}
Therefore, by the scaling properties in \eqref{EKhEK} and Lemma~\ref{lemScale}, we conclude that
\eqn{
|E_e(gw)|&\ls |e||E_{\hat{e}}(\hat{g}\hat{w})|\ls |e||\hat{g}|_{p,r,\hat{e}}\|\hat{w}\|_{0,\hat{e}}\ls h_e^p|e|^{\frac12-\frac1{r}}|g|_{p,r,e}\|w\|_{0,e}.}
That is, \eqref{gw1b} holds. This completes the proof of the lemma.
\end{proof}
% \begin{remark}
% % Given that the boundary is piecewise linear, the numerical integration on it can be performed using a format from a lower dimension. However, it is important to note that due to the fact that $V_h$ possesses only $H^1$ regularity, we must locally transform the dependent terms $\norm{v_h}_{p,e}$ and $\norm{w_h}_{1,e}$ into $\norm{v_h}_{p,K_e}$ and $\norm{w_h}_{1,K_e}$, by the inverse inequality, here $K_e$ represents one element containing $e$ . Therefore, the algebraic precision of the numerical integration format on $e$ should be at least $2p-1$, 
% \cb{Example~\ref{exm_3D} and Example~\ref{exm_2D} can demonstrate that the assumptions can not be weakened.} 
% \end{remark}
Thus far, we have meticulously scrutinized each term in equation \eqref{Strang1} individually. Now, Drawing upon the preceding lemmas, we present one of the central outcomes of this section.

\begin{theorem}\label{H1-err}
  Suppose that $ E_{\hat K}(\hat{\varphi})=0\;\forall \hat{\varphi}\in P_{2p-2}(\hat{K})$ and $E_{\hat{e}}(\hat\xi)=0\;\forall \hat\xi\in P_{2p-1}(\hat{e})$. Additionally, let $\al,\beta\in W^{p,\infty}(\Omega)$, $\gamma\in W^{p,\infty}(\Ga_2)$,  $f\in W^{p,q}(\Omega)$ for some $q\ge2$ satisfying $p>\frac{n}{q}$, and $g\in W^{p,r}(\Ga_2)$ for some $r\ge2$ satisfying $p>\frac{n-1}{r}$. Let $u\in H^{p+1}(\Om)$ be the solution to the problem \eqref{EP} and let $u_h^*$ be the solution to the discrete problem \eqref{integrationFEM}. Then we have
  \eqn{%\label{u-uh*1}
    \norm{u-u_{h}^{\star}}_{1,\Om}\lesssim h^p\big(&\abs{u}_{p+1,\Om}+(\norm{\al}_{p,\infty,\Om}+\norm{\beta}_{p,\infty,\Om})\norm{u}_{p,\Om}\\
    &+\norm{\gamma}_{p,\infty,\Ga_2}\norm{u}_{p,\Ga_2}+\norm{f}_{p,q,\Om}+\norm{g}_{p,r,\Ga_2}\big).}
\end{theorem}
\begin{proof}
  Using Lemmas \ref{apvpvpw-fw} and \ref{lem:gw}, the Cauchy-Schwarz inequality, and the trace inequality ($\norm{v}_{0,\pa\Om}\ls\norm{v}_{1,\Om}$), we have for any $v_h, w_h\in V_h$,
  \begin{align*}
    &|a(v_h,w_h)-a^{\star}(v_h,w_h)|\\
    &\ls \sum_{K\in\cTh}|E_K(\al\nabla v_h\cdot\nabla w_h)|+\sum_{K\in\cTh}|E_K(\beta v_h w_h)|+\sum_{e\in\mathcal{F}_h}|E_e(\ga v_h w_h)|\\
    &\ls h^p\big(\norm{\al}_{p,\infty,\Om}\norm{v_h}_{p,\cTh}+\norm{\beta}_{p,\infty,\Om}\norm{v_h}_{p,\cTh}+\norm{\gamma}_{p,\infty,\Ga_2}\norm{v_h}_{p,\mathcal{F}_h}\big)\|w_h\|_{1,\Om}.
  \end{align*}
 Likewise, from \eqref{EKfw}, \eqref{gw1b}, and the H\"older inequality, we conclude that
  \begin{align*}
    |F(w_h)-F^{\star}(w_h)|&\ls \sum_{K\in\cTh}|E_K(fw_h)|+\sum_{e\in\mathcal{F}_h}|E_e(gw_h)|\\
    &\ls h^p|\Om|^{\frac12-\frac1q}\norm{f}_{p,q,\Om}\norm{w_h}_{1,\Om}+h^p|\Ga_2|^{\frac12-\frac1{r}}\norm{g}_{p,r,\Ga_2}\norm{w_h}_{0,\Ga_2}\\
    &\ls h^p\big(\norm{f}_{p,q,\Om}+\norm{g}_{p,r,\Ga_2}\big)\norm{w_h}_{1,\Om}.
  \end{align*}
In addition, let $v_h=\Pi_h u\in V_h$ to be the Scott-Zhang interpolant \cite{sz90} of $u$. From the stability properties of the Scott-Zhang interpolation \cite[Theorem~3.1]{sz90}, we have  $\norm{v_h}_{p,\cTh}\ls\norm{u}_{p,\Om}$. According to \cite{sz90}, the Scott-Zhang interpolation operator can be appropriately defined such that the restriction of $\Pi_h u$ to $\Ga_2$ remains an  interpolation of $u|_{\Ga_2}$,  while preserving stability properties similarly to those of the classical Cl\'ement interpolation \cite{Clement75}. In particular,  $\norm{v_h}_{p,\mathcal{F}_h}\ls\norm{u}_{p,\Ga_2}$.
 By combining the aforementioned analyses with Lemmas \ref{Vh-ellipticity} and \ref{Strang} and the error estimate of $\Pi_h$ \cite[Theorem~4.1]{sz90}, we complete the proof.
\end{proof}

  \begin{remark} For the case of homogeneous Dirichlet boundary (when $\beta=0,\ga=0, g=0, \Ga_2=\emptyset$), the following estimate is given in \cite[Theorem 4.1.6]{Ci78}:
\eqn{\norm{u-u_{h}^{\star}}_{1,\Om}\lesssim h^p\big(\norm{\al}_{p,\infty,\Om}\norm{u}_{p+1,\Om}+\norm{f}_{p,q,\Om}\big).}
Our result extend this estimate to the elliptic problem with more general mixed boundary condition. In particular, when reduced to the case of Dirichlet boundary, our result becomes
\eqn{\norm{u-u_{h}^{\star}}_{1,\Om}\lesssim h^p\big(\abs{u}_{p+1,\Om}+\norm{\al}_{p,\infty,\Om}\norm{u}_{p,\Om}+\norm{f}_{p,q,\Om}\big)} 
 For certain multiscale problems involving oscillatory coefficients and solutions, our result can give better estimate. Our key idea for such an improvement is to take $v_h$ to be the Scott-Zhang interpolant of $u$ instead of finite element interpolant $I_hu$ used in \cite{Ci78}. 
  \end{remark}
%   \begin{remark}\label{insufficientPrecision}
%   In contrast to the numerical integration on $e$, the algebraic precision of the numerical integration on $K$ must not only satisfy the consistency error requirements for $a^\star(\cdot,\cdot)$ and $F^\star(\cdot)$, but also ensure the coercivity of $a^\star(\cdot,\cdot)$. When the algebraic precision of the numerical integration on $K$ falls below $2p-2$, Lemma~\ref{Vh-ellipticity} highlights that the solution to \eqref{integrationFEM} may not be guaranteed to be unique, thus the numerical integration on $K$ has a greater effect on error than on $e$. \co{}
% \end{remark}

  % \begin{proof}
  %   Similar to  \cite[Theorem 4.1.6]{Ci78}, we have that
  %   \eq{\norm{u-u^{\star}_h}_{1,\Om}\lesssim \inf_{v_h\in V_h}\big(\norm{u-v_h}_{1,\Om}+h^p (\norm{\al}_{p,\infty,\Om}+\norm{\beta}_{p,\infty,\Om})(\sum_{K\in\cTh}\norm{v_h}^2_{p,\cb{K}})^\frac12+h^p\norm{f}_{p,q,\Om}\big).\label{ciarlrt-1}}
    
  % \end{proof}

\subsection{$\mathbf{L^2}$-error estimation}\label{errorEstimationL2}
%   In addition to the volume integration, the face integration also impacts the error estimation of the finite element method. In the following discussion, we use the Neumann boundary value condition as an example to investigate the volume integration and face integration conditions that do not influence the error estimations.

%   Consider the problem 
%   \be
%     \label{neumannBD}
%     \begin{cases}
%       -\nabla\cdot(\al\nabla u) + \beta u = f \quad &{\rm in}\quad \Om,\\
%       \al\nabla u\cdot\bbn + \gamma u= g \quad &{\rm on} \quad\Ga.
%     \end{cases}
%   \ee

%  Consequently, the FEM with numerical integration is given by: find $u_{2h}^{\star}\in V_h$, s.t.
% \be\label{integrationFEMneumann}
% a_2^{\star}(u_{2h}^{\star},v_h)=F_2^{\star}(v_h)\quad \forall v_h\in V_h,
% \ee   
% where 

We would like to mention that \cite{CR72} considered the combined effect of curved boundaries and numerical integration for the elliptic problems with homogeneous Dirichlet boundary condition, deriving the $L^2$-error estimate. Although their result is quite general and practical, it is rather complicated. In this section, we focus solely on the effect of numerical integration in FEM for elliptic problems with mixed boundary condition.    

  In order to use the Aubin-Nitsche technique, we consider the following dual problem,
  \be\label{dualityFEM}
    \begin{cases}
      \ba{ll}
      -\nabla\cdot(\al\nabla w)+\beta w=u_h-u_h^{\star}\quad &\text{in}\quad\Om,\\
      w=0\quad &\text{on}\quad\Ga_1,\\
      \alpha\nabla w\cdot\bbn+\gamma w=0\quad &\text{on}\quad\Ga_2.
      \ea
    \end{cases}
  \ee
Subsequently, we can give the following Strang lemma in $L^2$-norm.
  \begin{lemma}\label{strang2}
    Let $w_h$ be the FEM solution of \eqref{dualityFEM}, then we have
    \begin{align}\label{Strangl2}				
      \norm{u-u_h^{\star}}_{0,\Om}\le\norm{u-u_h}_{0,\Om}+\frac{|F(w_h)-F^{\star}(w_h)|}{\norm{w}_{2,\Om}}+\frac{|a^{\star}(u_h^{\star},w_h)-a(u_h^{\star},w_h)|}{\norm{w}_{2,\Om}}.
    \end{align}
  \end{lemma}
  \begin{proof}
It can be seen from the regularity of the solution that $\norm{w}_{2,\Om}\lesssim\norm{u_h-u_h^{\star}}_{0,\Om}$. Then from \eqref{dualityFEM}, thanks to the symmetry of the bilinear form $a(\cdot,\cdot)$ and the Galerkin orthogonality $a(w-w_h,v_h)=0\;\forall v_h\in V_h$, we have 
    \begin{align*}					\norm{u_h-u_h^{\star}}_{0,\Om}\norm{w}_{2,\Om}\ls\norm{u_h-u_h^{\star}}_{0,\Om}^2=&a(w,u_h-u_h^{\star})=a(w_h,u_h-u_h^{\star})\\
=&a(w_h,u_h)-a(w_h,u_h^{\star})+a^{\star}(u_h^{\star},w_h)-F^{\star}(w_h)\\
=&F(w_h)-F^{\star}(w_h)+a^{\star}(u_h^{\star},w_h)-a(u_h^{\star},w_h).
\end{align*}
% Let $I_h$ represent the finite element interpolation operator onto $V_h$.				Then we can use the inverse inequality to estimate the discrete $H^2$ norm of $w_h$, 	\begin{align*}
% \norm{w_h}_{2,\cTh}&\ls \norm{w}_{2,\Om}+\norm{w-I_h w}_{2,\cTh}+\norm{I_h w-w_h}_{2,\cTh}\\&\ls \norm{w}_{2,\Om}+h^{-1}\norm{I_h w-w_h}_{1,\cTh}\\
% &\ls \norm{w}_{2,\Om}+h^{-1}(\norm{I_h w-w}_{1,\cTh}+\norm{w-w_h}_{1,\cTh})\\
% &\ls \norm{w}_{2,\Om}.
% \end{align*}
%      Combing these inequalities, the inclusion \eqref{Strangl2} can be deduced.
This together with the triangle inequality completes the proof of the lemma.
\end{proof}

% In the lemma above, we leverage the dual problem to present the Strang lemma for $\|u_h - u_h^{\star}\|_{0,\Omega}$. Exploiting the property of $u_h$ \co{} as the solution to \eqref{dualityFEM} offers valuable assistance to our analysis, thus motivating our focus on $u_h - u_h^*$. Unlike the analysis process for $H^1$-error, here the denominator involves the \sout{discrete} $H^2$-norm of $w$, it is worth noting that the algebraic precision $2p-2$ is not fully exploited in estimating $|a(\cdot,\cdot)-a^{\star}(\cdot,\cdot)|$ in Theorem~\ref{H1-err} (although it is necessary for the overall estimation of $|u-u_h^{\star}|_{1,\Omega}$). Consequently, we can demonstrate that when $p \geq 2$, the algebraic precision of numerical integration required for both $L^2$-error and $H^1$-error remains the same, with only the regularity of the data increasing. However, when $p=1$, the techniques mentioned above are not applicable. Hence, an improvement in the algebraic precision of numerical integration becomes imperative to optimize the $L^2$-error.

Proceeding, the following lemma gives higher order numerical integration errors of volume terms. The proof is similar to that of \cite[Theorem 5]{CR72}, and we sketch it for readability.

\begin{lemma}\label{apvpvpw-fw-D}
% For any $K\in\cTh$, when $p=1$, suppose that $ E_{K}(\varphi)=0$ for all $\varphi\in P_{1}(K)$, $\al,\beta\in W^{2,\infty}(K)$, and $f\in W^{2,q}(K)$, where $q\ge2$. Then, for $v,w\in P_1(K)$, we have
% \begin{align}
% |E_K(\al\nabla v\cdot\nabla w)|&\ls h_K^2\norm{\al}_{2,\infty,K}|v|_{1,K}|w|_{1,K},\label{lemD1}\\
% |E_K(\beta vw)|&\lesssim h_K^{2}\norm{\beta}_{2,\infty,K}\norm{v}_{1,K}\norm{w}_{1,K},\label{lemD2}\\
% |E_K(fw)|&\ls h_K^{2}|K|^{\frac12-\frac1q}\norm{f}_{2,q,K}\norm{w}_{1,K}.\label{lemD3}
% \end{align}

Assume that $E_{K}(\varphi)=0\;\forall \varphi\in P_{l}(K)$  where $l=\max\{p,2p-2\}$ and that $\al,\beta\in W^{p+1,\infty}(K)$, $f\in W^{p+1,q}(K)$ where $q\in[1,\infty] $ satisfies $p-\frac{n}{q}>0$. Then for any $v,w\in P_p(K)$
\begin{align}
|E_K(\al\nabla v\cdot\nabla w)|&\ls h_K^{p+1}\norm{\al}_{p+1,\infty,K}\norm{v}_{p,K}\norm{w}_{2,K},\label{lemD4}\\
|E_K(\beta vw)|&\lesssim h_K^{p+1}\norm{\beta}_{p+1,\infty,K}\norm{v}_{p,K}\norm{w}_{2,K},\label{lemD5}\\
|E_K(fw)|&\ls h_K^{p+1}|K|^{\frac12-\frac1q}\norm{f}_{p+1,q,K}\norm{w}_{2,K}.\label{lemD6}
\end{align}
\end{lemma}
\begin{proof}
	We first consider the simple case of $p=1$. Since $2-\frac{n}{q}>0$ and $ E_{\hat K}(\hat{\varphi})=0\;  \forall \hat{\varphi}\in P_1(\hat{K})$, with Lemma~\ref{lemBL} we have
\begin{align*}
\abs{ E_{\hat K}(\hat{\al}\hat{v}\hat{w})} \lesssim\abs{\hat{\al}\hat{v}\hat{w}}_{2,\hat{K}}\lesssim\abs{\hat{\al}}_{2,\infty,\hat{K}}\norm{\hat{v}}_{0,\hat{K}}\norm{\hat{w}}_{0,\hat{K}},\quad &\forall v,w\in P_0(K),\\
\abs{ E_{\hat K}(\hat{f}\hat{w})} \lesssim\abs{\hat{f}\hat{w}}_{2,q,\hat{K}}\lesssim\abs{\hat{f}}_{2,q,\hat{K}}\norm{\hat{w}}_{0,\hat{K}}+\abs{\hat{f}}_{1,q,\hat{K}}\abs{\hat{w}}_{1,\hat{K}},\quad&\forall w\in P_1(K),
\end{align*} 	
in the same way, $\forall v,w\in P_1(K)$
\begin{align*}
  \abs{ E_{\hat K}(\hat{\beta}\hat{v}\hat{w})} \lesssim&\abs{\hat{\beta}\hat{v}\hat{w}}_{2,\hat{K}}\lesssim\abs{\hat{\beta}}_{2,\infty,\hat{K}}\norm{\hat{v}}_{0,\hat{K}}\norm{\hat{w}}_{0,\hat{K}}\\
  &+\abs{\hat{\beta}}_{1,\infty,\hat{K}}(\norm{\hat{v}}_{0,\hat{K}}\abs{\hat{w}}_{1,\hat{K}}+\abs{\hat{v}}_{1,\hat{K}}\norm{\hat{w}}_{0,\hat{K}})+\norm{\hat{\beta}}_{0,\infty,\hat{K}}\abs{\hat{v}}_{1,\hat{K}}\abs{\hat{w}}_{1,\hat{K}},
\end{align*}
where we have used the norm equivalence theorem of finite-dimensional spaces on $\hat{K}$ to derive the last inequalities. Then from \eqref{EKhEK} and the scaling argument (see Lemma~\ref{lemScale}), \eqref{lemD4}--\eqref{lemD6} hold for $p=1$. 
				% \[\abs{E_K(\al\nabla u_h^{\star}\cb{\cdot\nabla} w_h)}\lesssim h_{K}^2\abs{\al}_{2,\infty,K}\norm{\nabla u_h^{\star}}_{0,K}\norm{\nabla w_h}_{0,K},\]
				% \[\abs{E_K(fw_h)}\lesssim h_{K}^2\abs{K}^{\frac12-\frac1q}\norm{f}_{2,q,K}\norm{w_h}_{1,K}.\]
				% Hence, \eqref{dirichletL2-1} \cb{follows from Lemma~\ref{Strangl2} and} the H\"older \sout{Cauchy-Schwarz} inequality.

				Next, we turn our attention to the case where $p\ge 2$. Let $\hat{\varPi}_0$ be the orthogonal projection from $L^2(\hat{K})$ onto $P_0(\hat{K})$. For any $v,w\in P_{p-1}(K)$, we have
				\[ E_{\hat K}(\hat{\al}\hat{v}\hat{w})= E_{\hat K}(\hat{\al}\hat{v}\hat{\varPi}_0\hat{w})+ E_{\hat K}(\hat{\al}\hat{v}(\hat{w}-\hat{\varPi}_0\hat{w})).\]
				Since $ E_{\hat K}(\cdot)$ vanishes over the space $P_{p}(\hat{K})$ ($2p-2\ge p$ holds when $p\ge2$), and $\norm{\hat{\varPi}_0\hat{w}}_{0,\hat{K}}\le\norm{\hat{w}}_{0,\hat{K}}$, from Lemma~\ref{lemBL}, we have
				\begin{align*}
| E_{\hat K}(\hat{\al}\hat{v}\hat{\varPi}_0\hat{w})|&\lesssim |\hat{\al}\hat{v}\hat{\varPi}_0\hat{w}|_{p+1,\infty,\hat{K}}\lesssim|\hat{\al}\hat{v}|_{p+1,\infty,\hat{K}}\norm{\hat{\varPi}_0\hat{w}}_{0,\hat{K}}\\
&\ls |\hat{\al}\hat{v}|_{p+1,\infty,\hat{K}}\norm{\hat{w}}_{0,\hat{K}}\lesssim\bigg(\sum_{j=0}^{p-1}|\hat{\al}|_{p+1-j,\infty,\hat{K}}|\hat{v}|_{j,\hat{K}}\bigg)\norm{\hat{w}}_{0,\hat{K}}.
				\end{align*}
For a given $\hat{w}\in P_{p-1}(\hat{K})$, by the norm equivalence theorem,  the linear form $ E_{\hat K}((\hat{w}-\hat{\varPi}_0\hat{w})(\cdot))$ is bounded on $W^{p,\infty}(\hat{K})$: $| E_{\hat K}((\hat{w}-\hat{\varPi}_0\hat{w})(\cdot))|\ls  \norm{\hat{w}-\hat{\varPi}_0\hat{w}}_{0,\hat K}\norm{\cdot}_{p,\infty,\hat K}$, and it vanishes over $P_{p-1}(\hat{K})$, then with the Bramble-Hilbert lemma, we have
				\begin{align*}
					| E_{\hat K}(\hat{\al}\hat{v}(\hat{w}-\hat{\varPi}_0\hat{w}))|\ls |\hat{\al}\hat{v}|_{p,\infty,\hat{K}}\|\hat{w}-\hat{\varPi}_0\hat{w}\|_{0,\hat{K}}\lesssim\bigg(\sum_{j=0}^{p-1}|\hat{\al}|_{p-j,\infty,\hat{K}}|\hat{v}|_{j,\hat{K}}\bigg)|\hat{w}|_{1,\hat{K}}.
				\end{align*}
				Therefore, \eqref{lemD4} can be deduced with \eqref{EKhEK} and Lemma~\ref{lemScale}. Moreover, we can similarly derive \eqref{lemD5} by replacing $\hat{\varPi}_0$ with $\hat{\varPi}_1$, which is the orthogonal projection from $L^2(\hat{K})$ onto $P_1(\hat{K})$.

				Next, when $p\ge 2$, it is easy to see that there exists a number $\rho\ge 1$ satisfying $p-\frac{n}q>p-1-\frac{n}{\rho}>0$ so that the following embeddings hold:
\eqn{W^{p,q}(\hat{K})\hookrightarrow W^{p-1,\rho}(\hat{K})\hookrightarrow C^0(\hat{K}).%\label{embedding}
}
 In the similar way, $\forall \hat w\in P_p(\hat{K})$,
				\begin{align*}
					\abs{ E_{\hat K}(\hat{f}\hat{w})}&\le \abs{ E_{\hat K}(\hat{f}\hat{\varPi}_1\hat{w})}+\abs{ E_{\hat K}(\hat{f}(\hat{w}-\hat{\varPi}_1\hat{w}))}\\
					&\ls \abs{\hat{f}\hat{\varPi}_1\hat{w}}_{p+1,q,\hat{K}}+\abs{\hat{f}}_{p-1,\rho,\hat{K}}\norm{\hat{w}-\hat{\varPi}_1\hat{w}}_{0,\hat{K}}\\
					&\ls \abs{\hat{f}}_{p+1,q,\hat{K}}\norm{\hat{\varPi}_1\hat{w}}_{0,\hat{K}}+\abs{\hat{f}}_{p,q,\hat{K}}\abs{\hat{\varPi}_1\hat{w}}_{1,\hat{K}}+\abs{\hat{f}}_{p-1,\rho,\hat{K}}\abs{\hat{w}}_{2,\hat{K}}\\
					&\ls \abs{\hat{f}}_{p+1,q,\hat{K}}\norm{\hat{w}}_{0,\hat{K}}+\abs{\hat{f}}_{p,q,\hat{K}}\abs{\hat{w}}_{1,\hat{K}}+(\abs{\hat{f}}_{p-1,q,\hat{K}}+\abs{\hat{f}}_{p,q,\hat{K}})\abs{\hat{w}}_{2,\hat{K}}.
				\end{align*}
				The last inequality holds because of $\norm{\hat{\varPi}_1\hat{w}}_{0,\hat{K}}\le\norm{\hat{w}}_{0,\hat{K}}$ and $\abs{\hat{\varPi}_1\hat{w}}_{1,\hat{K}}\le\abs{\hat{w}}_{1,\hat{K}}$. For the second, we know that from the equivalence of norms on finite dimensional spaces and the Poincar\'e inequality,
				\eqn{|\hat{\varPi}_1\hat{w}|_{1,\hat{K}} &=\inf_{q\in P_0(\hat K)}|\hat{\varPi}_1(\hat{w}-q)|_{1,\hat{K}}\ls \inf_{q\in P_0(\hat K)}\|\hat{\varPi}_1(\hat{w}-q)\|_{0,\hat{K}}\\
                &\ls \inf_{q\in P_0(\hat K)}\|\hat{w}-q\|_{0,\hat{K}}\lesssim|\hat{w}|_{1,\hat{K}}.}
				From \eqref{EKhEK} and Lemma~\ref{lemScale} again, we can get \eqref{lemD6}. 
\end{proof}

Furthermore, under enhanced data regularity conditions, the application of face quadrature formulas with \( 2p-1 \) algebraic precision leads to a one-order improvement in convergence rate compared to Lemma~\ref{lem:gw}.

\begin{lemma}\label{lem:gw2}
  Suppose that $E_{\hat{e}}(\hat{\varphi})=0\;\forall\hat{\varphi}\in P_{2p-1}(\hat{e})$. Then for any edge $e\in\mathcal{F}_h$, and $v,w\in P_p(e),$
  \begin{align}
    |E_e(\ga vw)|&\lesssim h_e^{p+1}\norm{\gamma}_{p+1,\infty,e}\norm{v}_{p,e}\norm{w}_{1,e}\quad \forall\gamma\in W^{p+1,\infty}(e),\label{gw1b-2}\\
  |E_e(gw)|&\lesssim h_e^{p+1}\norm{g}_{p+1,e}\norm{w}_{1,e}\quad\forall g\in H^{p+1}(e).\label{gw1b2}
  \end{align}
\end{lemma}
\begin{proof}
\eqref{gw1b-2} is proved in \cite[Lemma~3.4]{bbl11}. 
	We only prove \eqref{gw1b2}. Let  $\hat{\varPi}^{\hat{e}}$ as the orthogonal projection of $L^2(\hat{e})$ onto $P_1(\hat{e})$, then decompose $E_{\hat{e}}(\hat{g}\hat{w})$ into two parts,
	\[E_{\hat{e}}(\hat{g}\hat{w})=E_{\hat{e}}(\hat{g}\hat{\varPi}^{\hat{e}}\hat{w})+E_{\hat{e}}(\hat{g}(\hat{w}-\hat{\varPi}^{\hat{e}}\hat{w})).\] 

	We first consider the term $E_{\hat{e}}(\hat{g}\hat{\varPi}^{\hat{e}}\hat{w})$. Thanks to the Sobolev embedding theorem (see e.g. \cite[(1.4.6)]{bs08}), $H^{p+1}(\hat{e})$ $(\hat{K})\hookrightarrow C^0(\hat{e})$ holds, and hence $E_{\hat{e}}$ is bounded on $H^{p+1}(\hat{e})$. Noticing that $E_{\hat{e}}(\hat{\varphi})=0~\forall \hat\varphi\in P_{p}(\hat{e})$,  with Lemma~\ref{lemBL}, we have
	\[|E_{\hat{e}}(\hat{g}\hat{\varPi}^{\hat{e}}\hat{w})|\lesssim|\hat{g}\hat{\varPi}^{\hat{e}}\hat{w}|_{p+1,\hat{e}}.\]
	Further we have $\norm{\hat{\varPi}^{\hat{e}}\hat{w}}_{0,\hat{e}}\le\norm{\hat{w}}_{0,\hat{e}}$, and from the equivalence of norms on finite dimensional spaces and the Poincar\'e inequality,
	\eqn{|\hat{\varPi}^{\hat{e}}\hat{w}|_{1,\hat{e}} =\inf_{q\in P_0(\hat e)}|\hat{\varPi}^{\hat{e}}(\hat{w}-q)|_{1,\hat{e}}\ls \inf_{q\in P_0(\hat e)}\|\hat{\varPi}^{\hat{e}}(\hat{w}-q)\|_{0,\hat{e}}\ls \inf_{q\in P_0(\hat e)}\|\hat{w}-q\|_{0,\hat{e}}\lesssim|\hat{w}|_{1,\hat{e}},}
 	Thus, upon combining all our previous inequalities, we get
	\begin{align*}
		|E_{\hat{e}}(\hat{g}\hat{\varPi}^{\hat{e}}\hat{w})|&\le |\hat{g}|_{p+1,\hat{e}}\|\hat{\varPi}^{\hat{e}}\hat{w}\|_{0,\infty,\hat{e}}+|\hat{g}|_{p,\hat{e}}|\hat{\varPi}^{\hat{e}}\hat{w}|_{1,\infty,\hat{e}}\\
		&\le |\hat{g}|_{p+1,\hat{e}}\|\hat{\varPi}^{\hat{e}}\hat{w}\|_{0,\hat{e}}+|\hat{g}|_{p,\hat{e}}|\hat{\varPi}^{\hat{e}}\hat{w}|_{1,\hat{e}}\\
		&\ls |\hat{g}|_{p+1,\hat{e}}\|\hat{w}\|_{0,\hat{e}}+|\hat{g}|_{p,\hat{e}}|\hat{w}|_{1,\hat{e}}.
	\end{align*}

	It is noteworthy that when $p=1$, the difference $\hat{w}-\hat{\varPi}^{\hat{e}}\hat{w}$ becomes zero. Consequently, we can proceed with the assumption that $p\ge2$. From the Sobolev embedding theorem we have
  \eqn{ H^p(\hat{e})\hookrightarrow C^0(\hat{e}).}
  Proceeding with the familiar arguments, we find that 
	\[|E_{\hat{e}}(\hat{g}(\hat{w}-\hat{\varPi}^{\hat{e}}\hat{w}))|\le\|\hat{g}\|_{0,\infty,\hat{e}}\|\hat{w}-\hat{\varPi}^{\hat{e}}\hat{w}\|_{0,\infty,\hat{e}}\le||\hat{g}||_{p,\hat{e}}\|\hat{w}-\hat{\varPi}^{\hat{e}}\hat{w}\|_{0,\infty,\hat{e}}.\]
	Thus for a fixed $\hat{w}\in P_{p}(\hat{e})$, the linear functional $E_{\hat{e}}:~\hat{g}\rightarrow E_{\hat{e}}(\hat{g}(\hat{w}-\hat{\varPi}^{\hat{e}}\hat{w}))$ is bounded and vanishes over the space $P_{p-1}(\hat{e})$. Now using the Lemma~\ref{lemBL}, the equivalence of norms on finite dimensional spaces,  and the fact that $\hat{\varPi}^{\hat{e}}$ leaves $P_0(\hat{e})$ invariant, we can establish the following inequality:
	\eqn{|E_{\hat{e}}(\hat{g}(\hat{w}-\hat{\varPi}^{\hat{e}}\hat{w}))|&\lesssim|\hat{g}|_{p,\hat{e}}\|\hat{w}-\hat{\varPi}^{\hat{e}}\hat{w}\|_{0,\hat{e}}\lesssim|\hat{g}|_{p,\hat{e}}|\hat{w}|_{1,\hat{e}}.}
	%Here we also use the embedding inequality $W^{p+1,q}(\hat{e})\hookrightarrow W^{p,\rho}(\hat{2})$, which can be deduced similar to \eqref{embedding}.
	Combining all previous inequalities and  the equivalence of norms on finite dimensional spaces, 
	and hence the proof of the lemma is concluded by using Lemma~\ref{lemScale}.
\end{proof}

In addition to the consistency error of Lemma~\ref{strang2} needing to achieve an order of $p+1$ (so as not to affect the $L^2$ error), the conditions of Theorem~\ref{H1-err} required for the stability estimate of $ \norm{u_{h}^\star}_{p,\cTh} $ must also be satisfied. Therefore, we encapsulate these specific conclusions in the following theorem.

\begin{theorem}\label{L2-err}
  We assume that $ E_{\hat K}(\hat{\varphi})=0$\; $\forall\hat{\varphi}\in P_{l}(\hat{K})$, where $l=\max\{p,2p-2\}$, and that $E_{\hat{e}}(\hat\xi)\;\forall \hat\xi\in P_{2p-1}(\hat{e})$. Additionally, let $\al$, $\beta\in W^{p+1,\infty}(\Omega)$, $\gamma\in W^{p+1,\infty}(\Ga_2)$, $f\in W^{p+1,q}(\Omega)$ for some $q\ge2$ satisfying $p>\frac{n}{q}$, and $g\in H^{p+1}(\Ga_2)$. Then, if $u\in H^{p+1}(\Om)$,  we have
  \begin{align*}
    \norm{u-u_{h}^{\star}}_{0,\Om}&\lesssim h^{p+1}\big(1+h(\norm{\al}_{p+1,\infty,\Om}+\norm{\beta}_{p+1,\infty,\Om})+h^\frac12\norm{\ga}_{p+1,\infty,\Ga_2}\big)\\
    &\quad\times\big(\abs{u}_{p+1,\Om}+(\norm{\al}_{p+1,\infty,\Omega}+\norm{\beta}_{p+1,\infty,\Omega})\norm{u}_{p,\Om}+\norm{\ga}_{p+1,\infty,\Ga_2}\norm{u}_{p,\Ga_2}\\
    &\qquad+\norm{f}_{p+1,q,\Om}+\norm{g}_{p+1,\Ga_2}\big).
  \end{align*} 
\end{theorem}
\begin{proof}
  % Let $w$ denote the solution to the dual problem of equation \eqref{dualityFEM}, and $w_{h}$ represent the Corresponding FE solution, where the right-hand side is given by $u_{h}-u_{h}^{\star}$, and the mixed boundary condition is homogeneous. \cb{Similar to \eqref{Strangl2}, we have
  % \begin{align*}				
  %   \norm{u_{h}-u_{h}^{\star}}_{0,\Om}\le\frac{|F(w_{h})-F^{\star}(w_{h})|}{\norm{w}_{2,\Om}}+\frac{|a^{\star}(u_{h}^{\star},w_{h})-a(u_{h}^{\star},w_{h})|}{\norm{w}_{2,\Om}}.
  % \end{align*}}
  % \cb{We have already analyzed the consistency error resulting from numerical integration with respect to $K$ in Theorem~\ref{apvpvpw-fw-D}. }
According to Lemma~\ref{strang2}, we first estimate the consistency error terms on $K$. From  Lemma~\ref{apvpvpw-fw-D} and the Cauchy-Schwarz inequality, we have
\begin{align}\label{EK_L2}
\sum_{K\in\cTh}|E_K(\al\nabla u_h^\star\cdot\nabla w_h)|&\ls h^{p+1}\norm{\al}_{p+1,\infty,\Omega}\norm{u_h^\star}_{p,\cTh}\norm{w_h}_{2,\cTh},\\
\sum_{K\in\cTh}|E_K(\beta u_h^\star w_h)|&\lesssim h^{p+1}\norm{\beta}_{p+1,\infty,\Omega}\norm{u_h^\star}_{p,\cTh}\norm{w_h}_{2,\cTh},\\
\sum_{K\in\cTh}|E_K(f w_h)|&\ls h^{p+1}|\Omega|^{\frac12-\frac1q}\norm{f}_{p+1,q,\Om}\norm{w_h}_{2,\cTh}.
\end{align}
Based on Lemma~\ref{lem:gw2}, we can further derive the consistency error estimates on $e$:
  \begin{align}
    \sum_{e\in\mathcal{F}_h}|E_e(g w_{h})|&\lesssim h^{p+1} |\Ga_2|^{\frac12-\frac1q}\norm{g}_{p+1,\Ga_2}\norm{w_{h}}_{1,\Ga_2},\\
    \sum_{e\in\mathcal{F}_h}|E_e(\ga u_{h}^\star w_{h})|&\lesssim h^{p+1} \norm{\ga}_{p+1,\infty,\Ga_2}\norm{u^\star_{h}}_{p,\mathcal{F}_h} \norm{w_{h}}_{1,\Ga_2}.\label{Ee_L2}
  \end{align}
Next we estimate the terms in the right hand sides of the above estimates involving $u_h^\star$ and $w_h$. Let $\Pi_hu\in V_h$ be the Scott-Zhang interpolant of $u$. From the inverse inequality,  \cite[Theorem~4.1]{sz90}, and Theorem~\ref{H1-err}, we have
\eqn{
\norm{u_h^{\star}-\Pi_hu}_{p,\cTh}&\ls h^{1-p}\norm{\Pi_h u-u_h^{\star}}_{1,\Om}\ls h^{1-p}\big(\norm{u-\Pi_h u}_{1,\Om}+\norm{u-u_h^{\star}}_{1,\Om}\big)\notag\\
&\ls \abs{u}_{p,\Om}+h^{1-p}\norm{u-u_h^{\star}}_{1,\Om}.
}
Similarly, we have
\eqn{
\norm{u_h^{\star}-\Pi_hu}_{p,\mathcal{F}_h}&\ls h^{\frac12-p}\norm{\Pi_h u-u_h^{\star}}_{\frac12,\Ga_2}\ls h^{\frac12-p}\big(\norm{u-\Pi_h u}_{\frac12,\Ga_2}+\norm{u-u_h^{\star}}_{1,\Om}\big)\notag\\
&\ls \abs{u}_{p,\Ga_2}+h^{\frac12-p}\norm{u-u_h^{\star}}_{1,\Om}.
}
Therefore, by the triangle inequality and the stability of the Scott-Zhang interpolation \cite[Theorem~3.1]{sz90}, we conclude that
\eq{\label{uh*p}
\norm{u_h^{\star}}_{p,\cTh}\ls  \norm{u}_{p,\Om}+h^{1-p}\norm{u-u_h^{\star}}_{1,\Om}\quad\text{and}\quad \norm{u_h^{\star}}_{p,\mathcal{F}_h}\ls \norm{u}_{p,\Ga_2}+h^{\frac12-p}\norm{u-u_h^{\star}}_{1,\Om}.}
In the same way, we can get 
\eq{\label{wh2Om}\norm{w_h}_{2,\cTh}&\ls\norm{\Pi_h w}_{2,\cTh}+h^{-1}\norm{w_h-\Pi_h w}_{1\Om} \ls \norm{w}_{2,\Om},\\
   \norm{w_{h}}_{1,\Ga_2}&\lesssim \norm{\Pi_h w}_{1,\Ga_2} + h^{-\frac12}\norm{w_{h}-\Pi_h w}_{1,\Om} \lesssim\norm{w}_{2,\Om}. \label{wh1Ga}
.}
  Then combining with Lemma~\ref{strang2},  the estimates \eqref{EK_L2}--\eqref{wh1Ga}, and \eqref{neumann-uh-error}, we obtain
\eqn{
\norm{u-u_h^\star}_{0,\Om}&\ls\norm{u-u_h}_{0,\Om}+h^{p+1}\big(\norm{f}_{p+1,q,\Om}+\norm{g}_{p+1,\Ga_2}\\
&\quad+(\norm{\al}_{p+1,\infty,\Omega}+\norm{\beta}_{p+1,\infty,\Omega})\norm{u_h^\star}_{p,\cTh}+\norm{\ga}_{p+1,\infty,\Ga_2}\norm{u^\star_{h}}_{p,\mathcal{F}_h}\big)\\
&\ls h^{p+1}\big(\abs{u}_{p+1,\Om}+\norm{f}_{p+1,q,\Om}+\norm{g}_{p+1,\Ga_2}\\
&\qquad\qquad+(\norm{\al}_{p+1,\infty,\Omega}+\norm{\beta}_{p+1,\infty,\Omega})\norm{u}_{p,\Om}+\norm{\ga}_{p+1,\infty,\Ga_2}\norm{u}_{p,\Ga_2}\big)\\
&\quad +\big(h(\norm{\al}_{p+1,\infty,\Omega}+\norm{\beta}_{p+1,\infty,\Omega})+h^\frac12\norm{\ga}_{p+1,\infty,\Ga_2}\big) \norm{u-u_h^{\star}}_{1,\Om},}
which together with Theorem~\ref{H1-err} completes the proof of the theorem.
\end{proof}

\begin{remark}
\cite{CR72} derived an $L^2$ error estimate for the elliptic problems with homogeneous Dirichlet boundary condition on curved boundaries. For the case of a flat polygonal domain, the estimate simplifies to:
\eqn{\norm{u-u_{h}^{\star}}_{0,\Om}\ls h^{p+1}\norm{u}_{p+3,q,\Om},}
which imposes stronger regularity requirements on the exact solution $u$ and the dependences on the coefficients were not explicitly given there. 
 
\end{remark}

\section{Lower bounds}\label{lowerBounds}
In this section, we discuss the sharpness of the precision assumption for the quadrature rules in $L^2$-error estimates of the linear FEM.
Theorem~\ref{L2-err} shows that when $ p = 1 $, the situation is special, that is, the numerical volume integration needs to be one order more accurate to achieve the optimal convergence rate for the $ L^2 $-error compared to the $ H^1 $-error. Subsequently, one-dimensional examples are constructed for $\Ga_2\neq\emptyset$ (mixed boundary condition) and $\Ga_2=\emptyset$ (pure Dirichlet boundary condition) to demonstrate this assumption is necessary. 

\subsection{Mixed boundary condition}
In the domain $\Omega=[0,1]$, we consider the two points boundary value problem:
\be
    \begin{cases}
      &-u''=-6x\quad{\rm in}\quad \Omega,\\
      &u(0)=0,\quad u(1)+u'(1)=0,
    \end{cases}
  \ee
with the exact solution $u=x^3-2x$. Take  $\cTh=\set{[x_{i-1},x_i],i=1,\cdots,n}$, where $x_i=ih$, $h=\frac1n$. Denote by $\{\phi_i\}_{i=1}^n$ the corresponding nodal basis functions. Since $u(0)=0$, let $u_h^\star=u_1\phi_1+\cdots+u_n\phi_n$ be the solution of the FEM with numerical integration, and $U=[u_1,\cdots,u_n]^T$ be the solution vector. If we adopt the left-endpoint rectangle rule (with 0-th order algebraic precision) on each element, the stiffness matrix $A\in\mathbb{R}^{n\times n}$ and the right vector $F\in\mathbb{R}^{n\times 1}$ can be obtained and simplified as 
  \begin{equation*}
    A=
    \begin{bmatrix}
      2&-1& & & \\
      -1&2&-1& & \\
       &-1&\ddots&\ddots& \\
       & &\ddots&2&-1\\
        & & &-1&1+h
    \end{bmatrix}
    \qquad F=-6h^3
    \begin{bmatrix}
      1\\
      2\\
      \vdots\\
      n-1\\
      0
    \end{bmatrix}
  \end{equation*}
  Next, we perform LU decomposition of $A$, noted as $A=LR$.
  % \begin{equation*}
  % 	A:=LR=
  % 	\begin{bmatrix}
  % 		2& & & & \\
  % 		-1&\frac32& & & \\
  % 		 &-1&\ddots& & \\
  % 		 & &\ddots&\frac{n}{n-1}& \\
  % 		 & & &-1&2h
  % 	\end{bmatrix}
  % 	\begin{bmatrix}
  % 		1&-\frac12& & & \\
  % 		 &1&-\frac23& & \\
  % 		 & &\ddots&\ddots& \\
  % 		 & & &1&-\frac{n-1}{n}\\
  % 		 & & & &1
  % 	\end{bmatrix}
  % \end{equation*}
  Let $Y$ be the intermediate result satisfying that $LY=F$, we can conclude that
  \[Y=\big[-3h^3,\cdots,-i(2i+1)h^3,\cdots,-(n-1)(2n-1)h^3,-\frac12(n-1)(2n-1)h^2\big]^T.\]
  However, the solution $U$ of $RU=Y$ is hard to be formulated. Fortunately, we can calculate the last two values,
  \[u_{n-1}=-\frac12(n-1)(n+1)(2n-1)h^3,\quad u_n=-\frac12(n-1)(2n-1)h^2.\]
  Let $I_h u$ be the linear interpolation of $u$, then 
  \[(u_h^\star-I_h u)|_{[1-h,1]}=\frac12 h(3-h)x,\]
  thus,
  \begin{align*}
    \norm{u_h^\star-I_h u}_{0,\Om}^2&\ge\int_{(n-1)h}^{1} \frac14 h^2(3-h)^2 x^2\,dx\\
    &=\frac{1}{12}h^3(3-h)^2(1+(n-1)h+(n-1)^2h^2)\\
    &\ge\frac13 h^3.
  \end{align*}
  This implies that 
  \eqn{\norm{u-u_h^\star}_{0,\Om}\ge \norm{I_hu-u_h^\star}_{0,\Om}-\norm{I_hu-u}_{0,\Om}\ge\frac{\sqrt{3}}{3}h^\frac32-Ch^2,}
that is, the $L^2$-error cannot reach optimality when we employ quadrature rules with $2p-2=0$-th order algebraic precision, while it is sufficient for achieving optimal order $H^1$ error estimate. 

\subsection{Dirichlet boundary condition}\label{Ga2-empty}
The case of Dirichlet boundary condition ($\Gamma_2=\emptyset$) is treated separately because it is hard to construct a counterexample on equidistant grids (as in previous subsection) for the quadrature rule requirement for the linear FEM ($p=1$) in Theorem~\ref{L2-err}. In fact our extensive numerical tests on equidistant grids indicate that even a quadrature rule with zero-th order algebraic precision on the reference element $\hat K$ does not impair the  optimal $L^2$ convergence rate of the linear FEM --- possibly due to certain error cancellation properties. Therefore, we attempt to construct a counterexample on non-equidistant, yet still quasi-uniform, grids as follows.

We consider the following two points boundary value problem, 
  \be
    \begin{cases}
      &-u''=-6x\quad{\rm in}\quad \Omega,\\
      &u(0)=0,\quad u(1)=0,
    \end{cases}
  \ee
 where $\Omega=[0,1]$ and the exact solution is $u=x^3-x$. Let  $\cTh=\set{[x_{i-1},x_i],i=1,\cdots,n}$ be a grid of $\Omega$ given by $h=\frac1n$ and $x_i=\frac{ih}{2}+\frac{i^2 h^2}{2}$. Denote by $\{\phi_i\}_{i=0}^n$ the corresponding nodal basis functions. Since $u(0)=u(1)=0$, let $u_h^\star=u_1\phi_1+\cdots+u_{n-1}\phi_{n-1}$ be the discrete solution of the linear FEM with numerical integration, and $U=[u_1,\cdots,u_{n-1}]^T$ be the solution vector. Then we adopt the left-endpoint rectangle rule, the stiffness matrix $A\in\mathbb{R}^{(n-1)\times (n-1)}$  can be precisely calculated and subjected to LU decomposition, i.e.
  \begin{align*}
    A=
    \begin{bmatrix}
      a_1&b_1& & \\
      b_1&a_2&\ddots& \\
       &\ddots&\ddots&b_{n-2}\\
        & &b_{n-2}&a_{n-1}
    \end{bmatrix}
    =
    \begin{bmatrix}
      c_1& & & \\
      b_1&c_2& & \\
       &\ddots&\ddots& \\
       & &b_{n-2}&c_{n-1}
    \end{bmatrix}
    \begin{bmatrix}
      1&d_1& & \\
       &1&\ddots& \\
       & &\ddots&d_{n-2}\\
       & & &1
    \end{bmatrix}
    :=LR,
  \end{align*}
  where $a_i=\frac{2}{h+(2i-1)h^2}+\frac{2}{h+(2i+1)h^2}$, $b_i=-\frac{2}{h+(2i+1)h^2}$, $c_i=\frac{2}{h+(2i+1)h^2}+\frac{2}{ih+i^2h^2}$ and $d_i=-\frac{i+i^2h}{i+1+(i+1)^2h}$.
  And the right vector $F=[f_i]_{(n-1)\times 1}$ is given by $f_i=-\frac32 h^2\big((1+h)i+(3h+h^2)i^2+2h^2i^3\big)$. Subsequently we can calculate the intermediate result $Y$ satisfying that $LY=F$ and obtain 
  \eqn{Y=&[y_i]_{(n-1)\times 1},\quad y_i=-\frac32 h^2\big((1+h)y_{1,i}+(3h+h^2)y_{2,i}+2h^2 y_{3,i}\big),\\
    &y_{1,i}=\frac{ih\big(1+(2i+1)h\big)\left(2 (2i+1)+3i(i+1)h\right)}{24(1+(i+1)h)},\\
    &y_{2,i}=\frac{ih\big(1+(2i+1)h\big)\left(15 i(i+1)+2(2i+1)(3i^2+3i-1)h\right)}{120(1+(i+1)h)},\\
    &y_{3,i}=\frac{ih\big(1+(2i+1)h\big)\left(2 (2i+1)(3i^2+3i-1)+5i(i+1)(2i^2+2i-1)h\right)}{120(1+(i+1)h)}.}
  Let $T=[t_i]_{(n-1)\times 1}$ be the exact solution vector, $t_i=(\frac{ih}{2}+\frac{i^2h^2}{2})^3-\frac{ih}{2}-\frac{i^2h^2}{2}$, then we have
  \[T-U=R^{-1}R(T-U)=R^{-1}(RT-Y).\]
  Let $RT=[s_i]_{(n-1)\times 1}$ and 
  \begin{align*}
    s_i=
    \begin{cases}
      -\frac18 h^3i(1+ih)\big(2i+1+(2i^2+2i+1)h\big)\big(1+(2i+1)h\big),\quad &i=1,\cdots,n-2,\\
      (\frac{(n-1)h}{2}+\frac{(n-1)^2h^2}{2})^3-\frac{(n-1)h}{2}-\frac{(n-1)^2h^2}{2},\quad &i=n-1.
    \end{cases}
  \end{align*}
  It is obvious that $R^{-1}$ is also upper triangular, for any $i\le j$, $R^{-1}(i,j)=\frac{i+i^2h}{j+j^2h}$. And  if $0<h\le\frac12$, then
  \[(RT-Y)_{n-1}=s_{n-1}-y_{n-1}=\frac{h^2}{80}(1-h)(3-h)(h^3 + 21h^2 - 69h + 31)> 0.\]
  Thus, we can deduce that for any $i=2,\cdots,n-2$,
  \begin{align*}
    &t_i-u_i=\sum_{j=i}^{n-1}\frac{i+i^2h}{j+j^2h}(s_j-y_j)\\
    &\ge \sum_{j=i}^{n-2}\frac{i+i^2h}{j+j^2h}\frac{h^4j\big(1+(2j+1)h\big)\big(10j-5+5h(3j^2-j-2)+h^2(6j^3-j^2-9j-1)\big)}{40\big(1+(j+1)h\big)}\\
    &\ge\sum_{j=i}^{n-2}\frac{i+i^2h}{1+jh}\frac{h^4(10j-5)}{40}\ge\frac{h^4(i+i^2h)}{16}\sum_{j=i}^{n-2}(2j-1)\\
    &=\frac{h}{16}(ih+i^2h^2)\big((n-2)^2h^2-(i-1)^2h^2\big)\\
    &\ge\frac{h}{8}x_i(1-2h)(1-2h-x_i), \quad\text{if } 0<h\le\frac13,
  \end{align*}
  where we have used $(1-2h)^2-(ih-h)^2\ge (1-2h)\big(1-2h-\frac12(ih+i^2h^2)\big)$ for  $0<h\le\frac13$ to derive the last inequality.
  Let $\varphi=x(1-2h-x)$ and $I_h \varphi$ be the linear interpolation of $\varphi$. We can get for $0<h\le\frac13$,
  \begin{align*}
    \norm{I_h u-u_h^\star}_{0,\Omega}&\ge \frac{h(1-2h)}{8}\norm{I_h\varphi}_{0,[2h,(n-2)h]}\\
    &\ge \frac{h(1-2h)}{8}\big|\norm{\varphi}_{0,[2h,1-2h]}-\norm{\varphi-I_h\varphi}_{0,[2h,1-2h]}\big|\\
    &\ge \frac{1}{50}h-Ch^2,
  \end{align*}
  where $C$ is independent of the mesh size $h$.
  This implies that 
  \eqn{\norm{u-u_h^\star}_{0,\Om}\ge \norm{I_hu-u_h^\star}_{0,\Om}-\norm{I_hu-u}_{0,\Om}\ge\frac{1}{50}h-Ch^2,}
  that is, the $L^2$-error cannot reach optimality when we employ this kind of quadrature rule with $0$-th order algebraic precision.

\section{Numerical Experiment}\label{numericalExperiment}

In this section, we will demonstrate the impact of integration on FEM error estimates for both the $H^1$- and $L^2$-norms through a series of numerical tests. Additionally, we will highlight that achieving the optimal order of convergence is not feasible when the data lacks the regularity as outlined in the preceding sections. These numerical experiments have been carried out using MATLAB.

Given the unique nature of \( p = 1 \), we use a separate example to verify the necessity of the conditions assumed in previous theories. For convenience of expression, denote by $ap_K$ and $ap_e$ the algebraic precision of numerical integration on $K$ and $e$, respectively.

\begin{example}\label{exm_2D} Consider Problem \eqref{EP} in two dimensions with $\Om$ to be the unit square $[0,1]\times[0,1]$,  $\Ga_1=\set{x\in[0,1],y=0~\text{or}~1}$,  and the data $\alpha=\beta=\gamma=1$. $f$ and $g$ are so chosen so that the exact solution  is given by,
  \begin{align*}
    u=\sin{\pi x}\sin{\pi y}.
  \end{align*}
\end{example}

\begin{figure}[!htbp]
  \centering
  \includegraphics[width=0.98\textwidth]{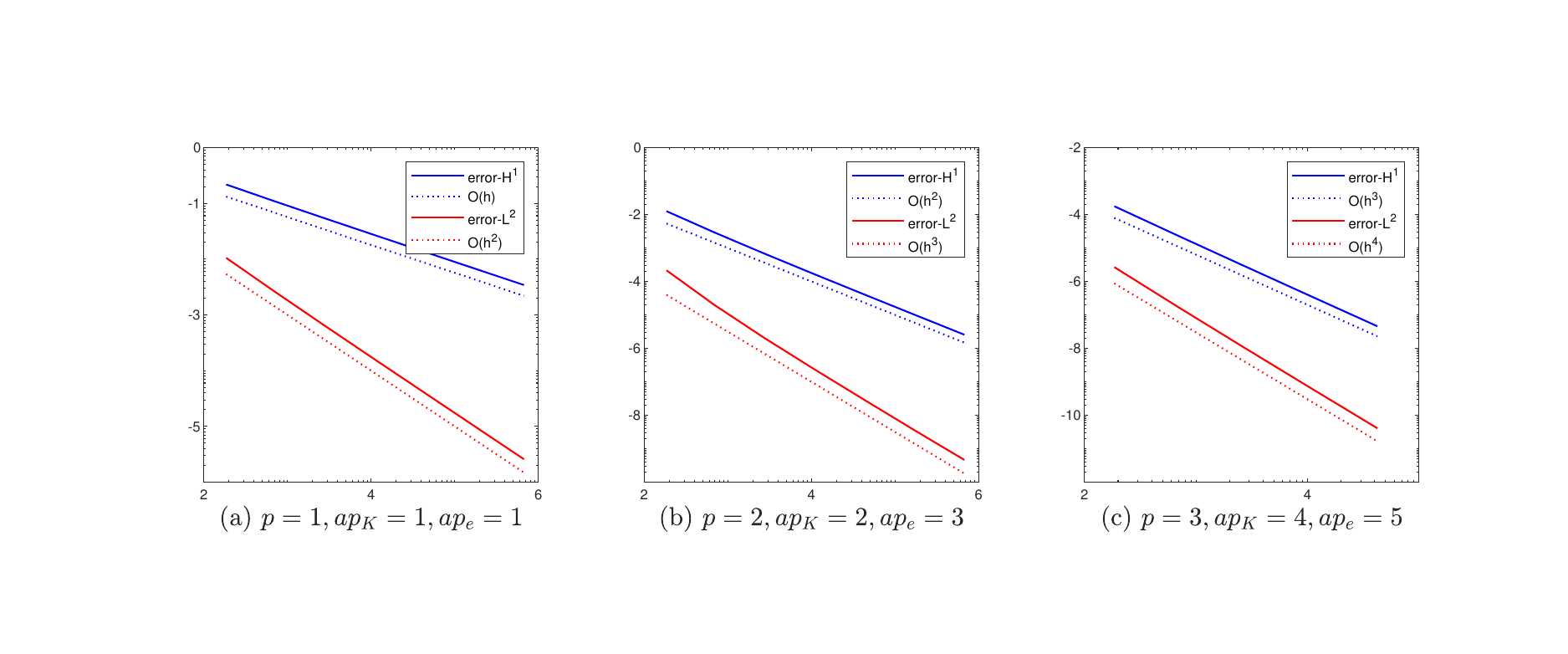}
  \caption{Sufficient algebraic precision of numerical integration.}
  \label{u1}
\end{figure}

In Example~\ref{exm_2D}, we take smooth exact solution $u$ and data functions to explore the impact of numerical integration on error estimates. The meshes are generated by applying one round of local refinement (on a random 25\% of the elements) to each level of a uniformly refined base mesh sequence, with uniformly bounded element size ratios across all levels. As shown in Figure~\ref{u1}, when $ap_K=\max\{p,2p-2\}$ and $ap_e=2p-1, p=1,2,3,$ satisfying Theorem~\ref{H1-err} and Theorem~\ref{L2-err}, the optimal $H^1$- and $L^2$-error convergence order can be achieved.

\begin{figure}[!htbp]
  \centering
  \includegraphics[width=0.98\textwidth]{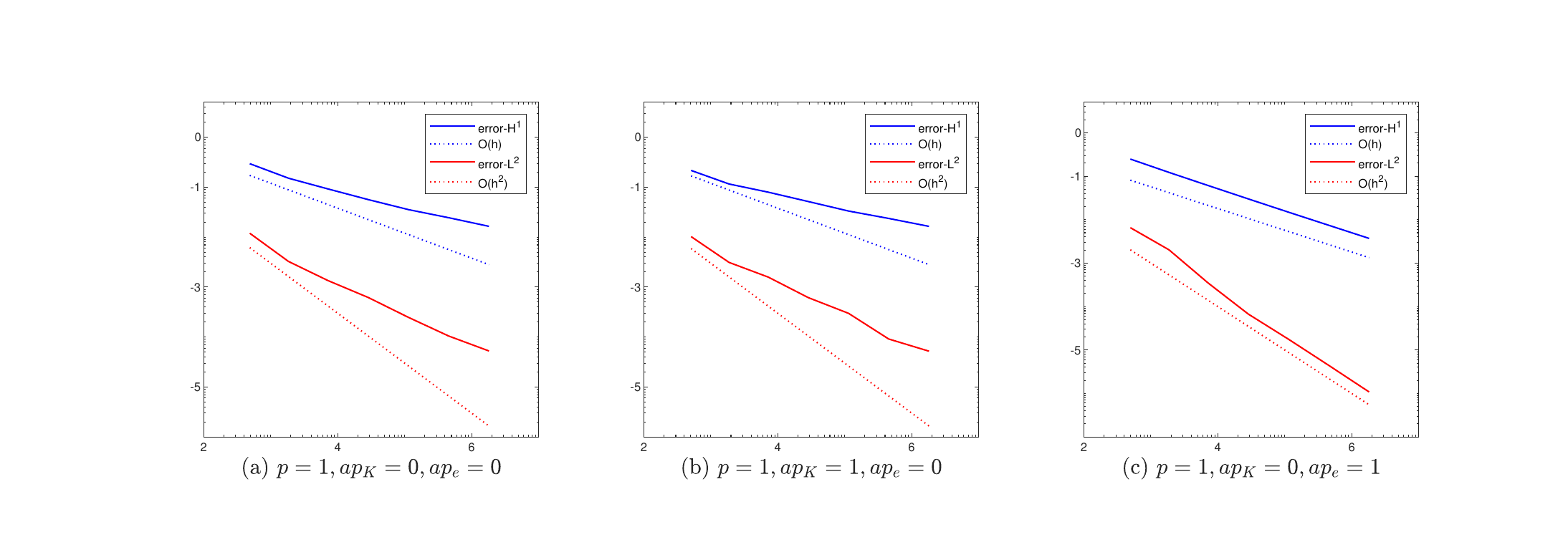}
  \caption{Insufficient algebraic precision of numerical integration for $p=1$.}
  \label{InsufficientAlgebraicPrecisionp1}
\end{figure}
\begin{figure}[!htbp]
  \centering
  \includegraphics[width=0.98\textwidth]{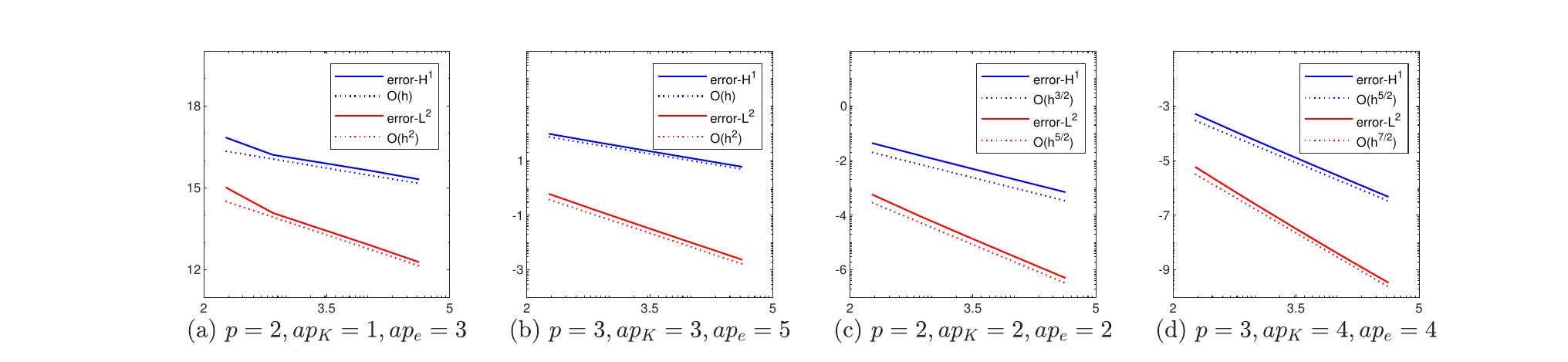}
  \caption{Insufficient algebraic precision of numerical integration for $p=2,3$.}
  \label{InsufficientAlgebraicPrecision}
\end{figure}

By the way, we first verify the necessity of the algebraic precision of numerical integration on $K$ and $e$ required by the theories for $p=2,3$. If \( ap_K \) is chosen as \( 2p - 3 \), then as depicted in Figures \ref{InsufficientAlgebraicPrecision}(a)--(b), neither the \( H^1 \)-error nor the \( L^2 \)-error attain the optimal convergence rate. Notably, when \( p = 2 \), the numerical solution is even erroneous, this phenomenon can be attributed to the loss of uniform ellipticity as established in Lemma~\ref{Vh-ellipticity}. Similarly, when \( ap_e \) is set to \( 2p - 2 \), it is observed in Figures \ref{InsufficientAlgebraicPrecision}(c)--(d) that the both convergences in $H^1$- and $L^2$-norms are suboptimal by loss of a half order. 

As for the case \(p=1\), it can be observed from Figures \ref{InsufficientAlgebraicPrecisionp1}(a)--(b) that $ap_e$ fails to satisfy the conditions of Theorem~\ref{H1-err} and Theorem~\ref{L2-err}, and accordingly, the corresponding errors do not attain the optimal convergence order. In Figure \ref{InsufficientAlgebraicPrecisionp1}(c), the \(ap_K=0\) case still exhibits second-order convergence in the \(L^2\) norm, which is likely caused by certain error cancellation. Furthermore, the necessity of first-order element quadrature precision required for the \(L^2\) error when \(p=1\) is verified in Section~\ref{lowerBounds} and illustrated in Figure \ref{3D}(b).
% As for the absence of this phenomenon when $p=1$ (Fig.~\ref{3D}(c)), we hypothesize that the underlying mechanism may be analogous to the scenario discussed in Subsection~\ref{Ga2-empty}.

\begin{example}\label{exm3}
We set $\Om=[0,1]\times[0,1]$, $\Ga_1=\emptyset$, $\alpha=\beta=1$, and $\ga=0$ in Problem \eqref{EP}, the exact solution is chosen as
  \begin{align*}
    u=-\frac{9}{88}(\sqrt{2}-r)^{\frac{11}{3}},~\text{where}~r=\sqrt{x^2+y^2}.
  \end{align*}
\end{example}

\begin{figure}[!htbp]
  \centering
  \includegraphics[width=0.6\textwidth]{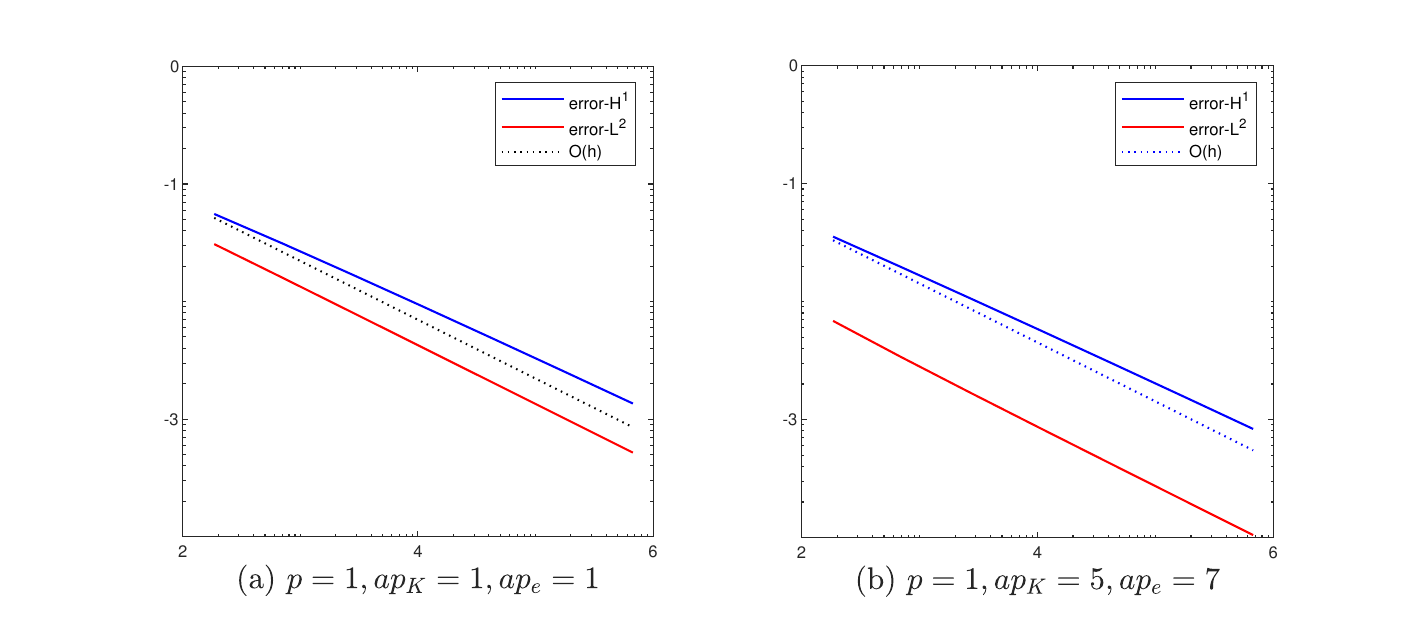}
  \caption{Insufficient data regularity.}
  \label{InsufficientRegularity}
\end{figure}

Next, we examine the necessity of data regularity in Example~\ref{exm3}. Given that $f$ does not belong to $L^2(\Om)$, we restrict our analysis to $p=1$, we take the algebraic precision conditions that happen to have no effect on both $H^1$- and $L^2$-errors ($ap_K=1$ and $ap_e=1$) or far exceed the theoretical results ($ap_K=5$ and $ap_e=7$). As a result in Figure~\ref{InsufficientRegularity}, we observe that the $H^1$-error convergence rate is slightly lower than the theoretical expectation, while the $L^2$-error experiences a significant deviation from the expected theoretical values. This shows that the requirements of data regularity are necessary. Moreover, although over-integration is helpful for accuracy of FE solution, the effect is limited.
 
\begin{example}\label{exm_3D} 
  Consider Problem \eqref{EP} in three dimensions with $\Om$ to be the unit cube $[-1,1]\times[-1,1]\times[0,2]$, $\alpha$, $\beta=1$, $\ga=0$, and $\Ga_1=\emptyset$. $f$ and $g$ are so chosen so that  the exact solution is,
  \begin{align*}
    u=(\frac13 x^3-\frac12 x^2)(y^3-3y)(z^3-3z^2).
  \end{align*}
\end{example}

 Finally, consider this three dimensional numerical example. We discretize the domain $\Om$ using the standard triangulations \cite{BE91} (also called Cube-VI-I mesh in \cite[Chapter 4]{ZY2023}).
  \begin{figure}
    \centering
    \includegraphics[width=0.98\textwidth]{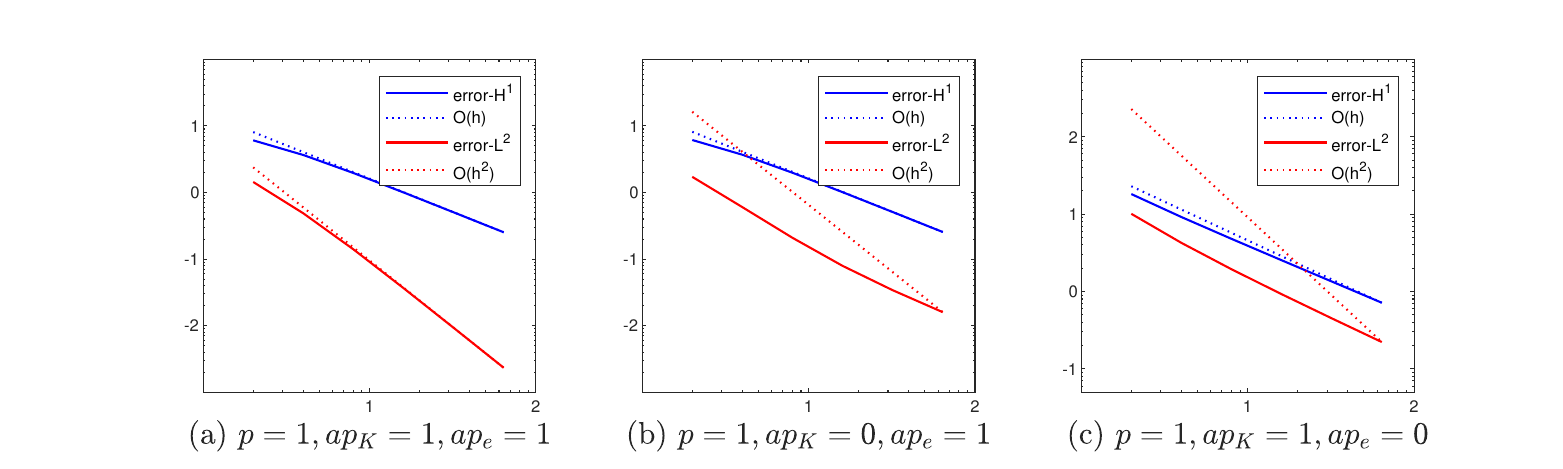}
    \caption{An 3D example with $p=1$.}\label{3D}
  \end{figure}

As depicted in Figure \ref{3D}(a), setting both $ap_K$ and $ap_e$ to $1$ enables both $H^1$- and $L^2$-errors to attain the optimal convergence rate. In Figure \ref{3D}(b), when $ap_K$ is set to $0$, $H^1$ achieves the optimal convergence rate, but $L^2$ does not. The Figure \ref{3D}(c) reveals that with $ap_e$ at $0$, neither $H^1$- nor $L^2$-error displays optimal convergence, with $H^1$-error exhibiting a slight order reduction. These phenomena are consistent with the theoretical results outlined in Theorem~\ref{H1-err} and Theorem~\ref{L2-err}.

\section{Conclusions}
% This paper investigates the influence of numerical integration accuracy on the p-th order finite element discretization of second-order elliptic problems with mixed boundary conditions. Rigorously derived sufficient conditions on the algebraic precision of element and boundary quadrature rules are presented to preserve the optimal \(H^1\) and \(L^2\) error convergence rates. We construct counterexamples for linear elements (\(p=1\)) to demonstrate that first-order accuracy for element quadratures is strictly necessary to achieve the optimal \(L^2\) error. Our results genuinely relax the regularity assumptions on the exact solution imposed by the classical results in \cite{Ciarlet1991,Ci78,CR72}. These results provide additional theoretical support for the optimal choice of quadrature rules in finite element computations. The analysis of numerical integration effects on rectangular elements, and even more broadly for interface problems, is deferred to future work.}

To summarize, this paper revisits the influence of numerical integration accuracy on the \(p\)-th order finite element discretization of second-order elliptic problems, and extends the classical results regarding Dirichlet boundary condition to  mixed Dirichlet and Robin boundary conditions. It derives sufficient conditions on the algebraic precision of element and boundary quadrature rules to preserve the optimal \(H^1\) and \(L^2\) convergence rates. The results are new for Robin boundary conditions. For the case of $L^2$-error and Dirichlet boundary conditions, our results relax the regularity requirements on the exact solution compared with classical results.  Furthermore, counterexamples for linear elements (\(p=1\)) are proposed to verify the necessity of the first-order precision of element quadrature for optimal \(L^2\) convergence.  Extensions to rectangular elements and interface problems will be addressed in future work.

\bibliographystyle{unsrt}
\bibliography{ref}

\end{document}